\documentclass[11pt,reqno]{amsproc}

\usepackage{amsmath,amsfonts,amssymb,amsthm}
\usepackage[abbrev,lite,nobysame]{amsrefs}
\usepackage{mathrsfs}
\usepackage{graphics,graphicx}
\usepackage[usenames,dvipsnames]{color}
\usepackage{times}
\usepackage{bbm}
\usepackage[margin=1in]{geometry}
 \usepackage{mathtools}

 \mathtoolsset{showonlyrefs=true}

\usepackage{enumerate}
\usepackage{hyperref}
 
\usepackage{bm}

\usepackage{ dsfont }

\usepackage{enumitem}
\usepackage{tikz-cd,pgfplots}
\usetikzlibrary{arrows, arrows.meta, decorations.pathmorphing}

\tikzset{snake it/.style={decorate, decoration=snake}}
\usetikzlibrary{arrows.meta}
\tikzset{%
  tipA/.tip={Triangle[open,angle=45:4pt]}
}

\title{Exponential mixing for random dynamical systems and an example of Pierrehumbert}
\author[A.\ Blumenthal]{Alex Blumenthal}
\address{School of Mathematics, Georgia Institute of Technology, Atlanta, GA 30332, USA}
\email{ablumenthal6@gatech.edu}

\author[M.\ Coti Zelati]{Michele Coti Zelati}
\address{Department of Mathematics, Imperial College London, London, SW7 2AZ, UK}
\email{m.coti-zelati@imperial.ac.uk}

\author[R.\ S.\ Gvalani]{Rishabh S. Gvalani}
\address{Max-Planck-Institut für Mathematik in den Naturwissenschaften, 04103 Leipzig, Germany.}
\email{gvalani@mis.mpg.de}

\date{\today}

\makeatletter
\@namedef{subjclassname@2020}{\textup{2020} Mathematics Subject Classification}
\makeatother


\thanks{AB was supported by National Science Foundation grant DMS-2009431. 
MCZ acknowledges funding from the Royal Society through a University Research Fellowship 
(URF\textbackslash R1\textbackslash 191492). MCZ would like
to thank the Isaac Newton Institute for Mathematical Sciences, Cambridge, for support and hospitality during
the programs \emph{Mathematical aspects of turbulence} and \emph{Frontiers in kinetic theory} 
where work on this
paper was undertaken.}

\newtheorem{thm}{Theorem}[section]
\newtheorem{prop}[thm]{Proposition}
\newtheorem{lem}[thm]{Lemma}
\newtheorem{cor}[thm]{Corollary}

\theoremstyle{definition}
\newtheorem{defn}[thm]{Definition}

\newtheorem{rmk}[thm]{Remark}
\newtheorem{cla}[thm]{Claim}

\numberwithin{equation}{section}

\newcommand{\C}{\mathbb{C}}

\newcommand{\E}{\mathbb{E}}

\newcommand{\N}{\mathbb{N}}
\renewcommand{\P}{\mathbb{P}}
\newcommand{\R}{\mathbb{R}}

\newcommand{\Z}{\mathbb{Z}}

\newcommand{\Bc}{\mathcal{B}}
\newcommand{\Fc}{\mathcal{F}}

\newcommand{\Ac}{\mathcal{A}}
\newcommand{\Sc}{\mathcal{S}}

\renewcommand{\Mc}{\mathcal M}

\renewcommand{\d}{\delta}
\newcommand{\e}{\epsilon}
\newcommand{\eps}{\varepsilon}

\newcommand{\de}{\partial}
\newcommand{\dd}{{\rm d}}
\newcommand{\ee}{{\rm e}}

\newcommand{\Id}{\operatorname{Id}}
\newcommand{\dist}{\operatorname{dist}}

\newcommand{\T}{\mathbb T}

\newcommand{\Leb}{\operatorname{Leb}}
\newcommand{\Lip}{\operatorname{Lip}}

\newcommand{\Supp}{\operatorname{Supp}}

\newcommand{\uo}{{\underline{\omega}}}
\newcommand{\Cor}{\operatorname{Cor}}

\begin{document}
\begin{abstract}
We consider the question of exponential mixing for random dynamical systems on arbitrary compact manifolds without boundary. 
We put forward a robust, dynamics-based framework that allows us to construct space-time smooth, uniformly bounded in time, universal exponential mixers.
The framework is then applied to the problem of proving exponential mixing in a classical example proposed by Pierrehumbert in 1994,  consisting 
of alternating periodic shear flows with randomized phases. 
This settles a longstanding open problem on proving the existence 
of a space-time smooth (universal) exponentially mixing incompressible velocity field on 
a two-dimensional periodic domain while also providing a toolbox for constructing such smooth universal mixers in all dimensions.
\end{abstract}

\maketitle

\tableofcontents

\section{Introduction}\label{sec:intro}

Mixing by incompressible flows is a fundamental stabilization mechanism in fluid mechanics, usually associated with the transfer of energy from large to small spatial scales, in a manner that is conservative and reversible for finite times but results in an irreversible loss of information in the long-time limit. In its simplest mathematical setting, this phenomenon is often associated with the study of the long-time dynamics of the transport equation
\begin{align}\label{eq:passivescal}
\de_t\rho +u\cdot\nabla \rho=0,
\end{align}
posed in a $d$-dimensional periodic domain $\T^d$ parametrized by $[0,2\pi)^d$. The unknown $\rho=\rho(t,x):[0,\infty)\times \T^d\to \R$ is advected by a specified divergence-free velocity field $u=u(t,x):[0,\infty)\times \T^d\to \R^d$ and can be assumed, without loss of generality, to be mean-free, namely
\begin{align}\label{eq:meanfree}
\int_{\T^d} \rho(t,x)\dd x=0, \quad \textrm{for all} \quad t\geq0.
\end{align}
In applications, the unknown 
$\rho$ denotes the concentration of a scalar (for e.g., a chemical in water), that passively evolves according to an assigned stirring mechanism. A commonly used measure of the mixing reached by $\rho$ at a given time is in terms of the homogeneous negative Sobolev norms 
\begin{align}\label{eq:negsob}
\|\rho(t)\|^2_{\dot{H}^{-s}} =\sum_{k\in \Z^d\setminus\{0\}} |k|^{-2s}|\rho_k(t)|^2,
\end{align}
where the $\rho_k$'s are the Fourier coefficients of $\rho$. This idea was introduced in \cite{MMP05}, and then revisited in \cite{LTD11}, for different values of the parameter $s>0$. Decay of \eqref{eq:negsob} as $t \to \infty$ implies that the $L^2$ mass of $\rho(t)$ is being evacuated to high modes, signaling mixing. 

The questions of how fast can the quantity in \eqref{eq:negsob} decay (given certain constraints on $u$), and which velocity fields $u$  can achieve such decay,  
have attracted considerable attention recently.  It is a standard fact that 
for an incompressible velocity field $u \in L^\infty_t W^{1,\infty}_x$, the quantity \eqref{eq:negsob} cannot decay faster than exponentially. For velocity fields in $L^\infty_t W^{1,p}_x, p \in (1,\infty)$, exponential lower bounds for \eqref{eq:negsob} were established in \cite{CDL08}; see also \cite{GKX14,Seis13}. 
The case $p=1$, corresponding to Bressan's rearrangement cost conjecture \cite{Bressan03}, is still open. 

The main focus of this work is providing \emph{upper} bounds for \eqref{eq:negsob}. Given a mean-free initial datum $\rho(0,\cdot)$, it was shown in \cite{YZ17} that there exists a velocity field $u\in L^\infty_tW^{1,p}_x$, for some $p\in(2,3)$, that mixes the corresponding solution exponentially fast. Different regularity constraints on $u$ were then studied in depth in \cite{ACM19}. It is important to notice that in these works, the velocity field $u$ heavily depends on the initial datum. The natural question of finding a universal mixer, namely a velocity field that mixes all mean-free initial data (with a certain regularity at least), was settled in \cite{EZ19}. As in \cite{YZ17}, such velocities have only limited space regularity, and they are quite complicated (for instance, they are not time-periodic). It turns out that universal exponential mixers can also be constructed from almost-every solution of the stochastic Navier--Stokes equations forced by 
noise which is sufficiently non-degenerate \cite{bedrossian2019almost}. Although these flows are stationary in time and locally space-time regular, such flows are not periodic nor uniformly bounded in time.

In this paper, we establish the existence of space-time smooth, uniformly bounded in time, universal exponential mixers on periodic domains. This manuscript contains: 
\begin{itemize}
	\item[(a)] a robust, dynamics-based framework that allows us to construct smooth, uniformly bounded in time, universal exponential mixers on arbitrary compact boundaryless manifolds (Sections \ref{sec:rdsPrelims}, \ref{sec:LE} and \ref{sec:mixing}); and
	\item[(b)] an application of this framework to the classical two-dimensional example of Pierrehumbert \cite{pierrehumbert1994tracer}, consisting of alternating periodic shear flows with randomized phases (Section \ref{sec:pierrehumbert}). 
\end{itemize}
Mixing properties of the Pierrehumbert model have been studied extensively in the applied and computational literature  (see, e.g., \cite{PY06,Thiffeault12,TDG04} and references therein). Numerical evidence suggests it is a universal exponential mixer \cite{CRFWZ21}, although to the authors' best  knowledge ours is the first rigorous proof of this assertion. 
 

\subsection{The Pierrehumbert model}\label{sub:PierreModel}
A natural candidate for an exponentially mixing flow on the two-dimensional torus $\T^2$ was introduced by Pierrehumbert in \cite{pierrehumbert1994tracer}. In this model, the velocity field $u$ alternates, after every time interval of size $\tau>0$, between two transversal shears with a randomly and independently chosen phase shift.
More precisely, let $\{\omega_j=(\omega^1_j,\omega^2_j)\}_{j\in\N} \subset \R^2$  be a sequence of i.i.d random variables uniformly distributed in $[0,2\pi)^2$. 
For $(t, x) \in [0,\infty) \times \T^2$, we define
\begin{align}\label{eq:velbuild1}
u(t, x_1, x_2) := 
\begin{pmatrix} \sin(x_2 - \omega_n^1) \\ 0 \end{pmatrix} 
\end{align}
if $t \in [(2n-2) \tau, (2n - 1) \tau)$ for some $n \in \N$, and 
\begin{align}\label{eq:velbuild2}
u(t, x_1, x_2) := 
\begin{pmatrix} 0 \\\sin(x_1 - \omega_n^2) \end{pmatrix} 
\end{align}
for $t \in [(2 n -1) \tau, 2 n \tau)$. 


Intuitively,  we alternate between the two sinusoidal shears, starting without loss of generality with a horizontal shear, while simultaneously, randomly and independently, changing the phases of each of the shears. The intriguing fact is that, while shear flows on the torus are mixing with at most at an asymptotic rate of $1/\sqrt{t}$ \cite{BCZ17}, the combination of vertical and horizontal shearing can speed up the mixing considerably.
\begin{figure}
\centering
\begin{minipage}{0.3\textwidth}
\includegraphics[scale=0.5]{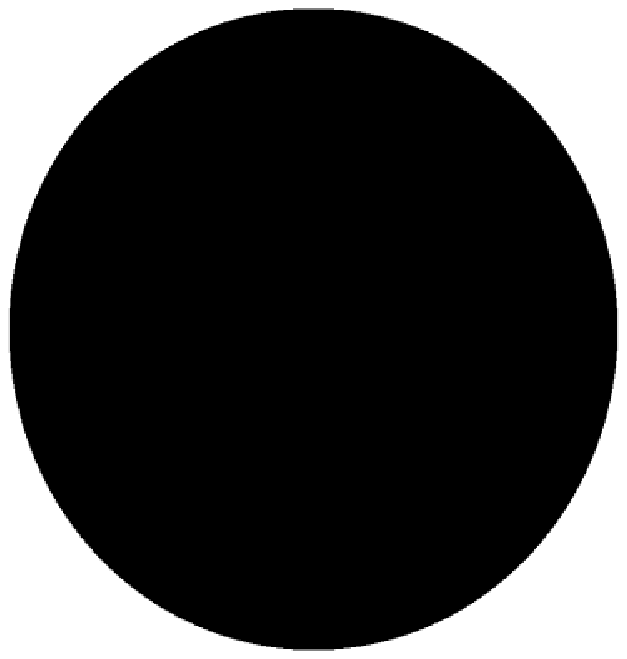}
\end{minipage}
\begin{minipage}{0.3\textwidth}
\includegraphics[scale=0.5]{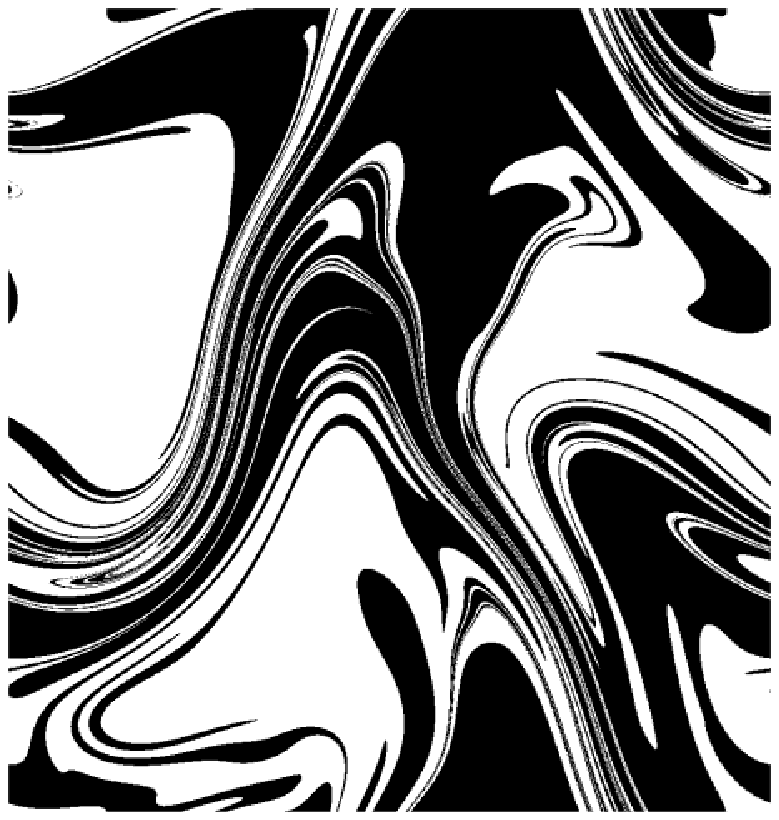}
\end{minipage}
\begin{minipage}{0.3\textwidth}
\includegraphics[scale=0.5]{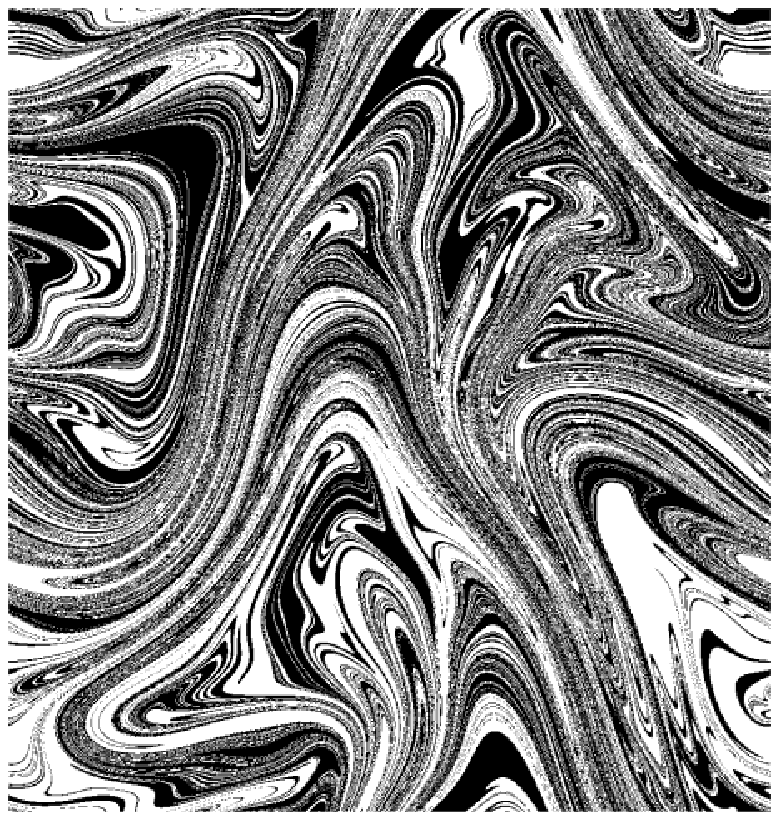}
\end{minipage}
\begin{minipage}{0.3\textwidth}
\includegraphics[scale=0.5]{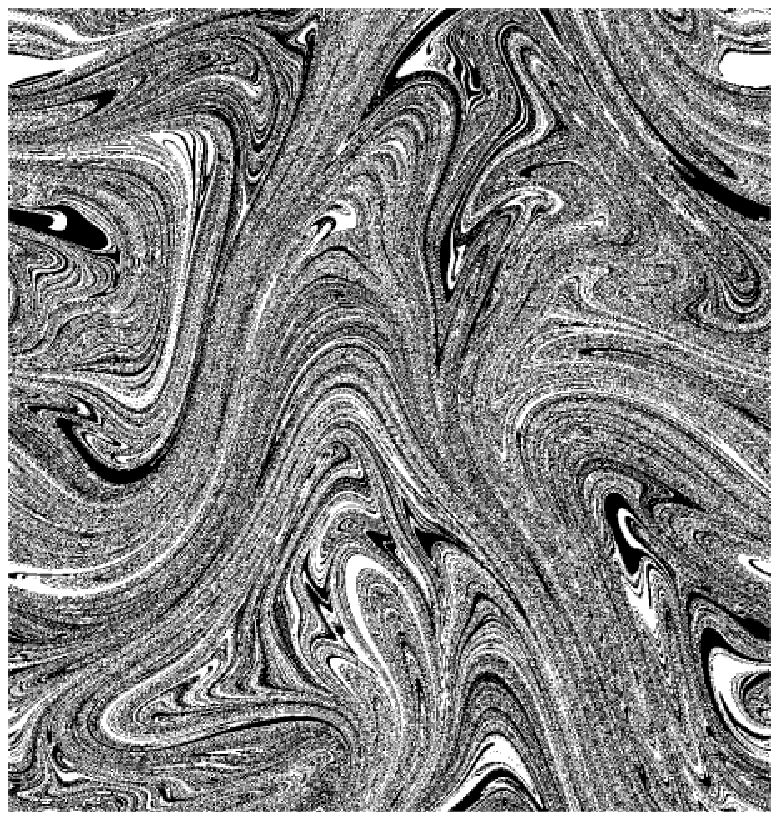}
\end{minipage}
\begin{minipage}{0.3\textwidth}
\includegraphics[scale=0.5]{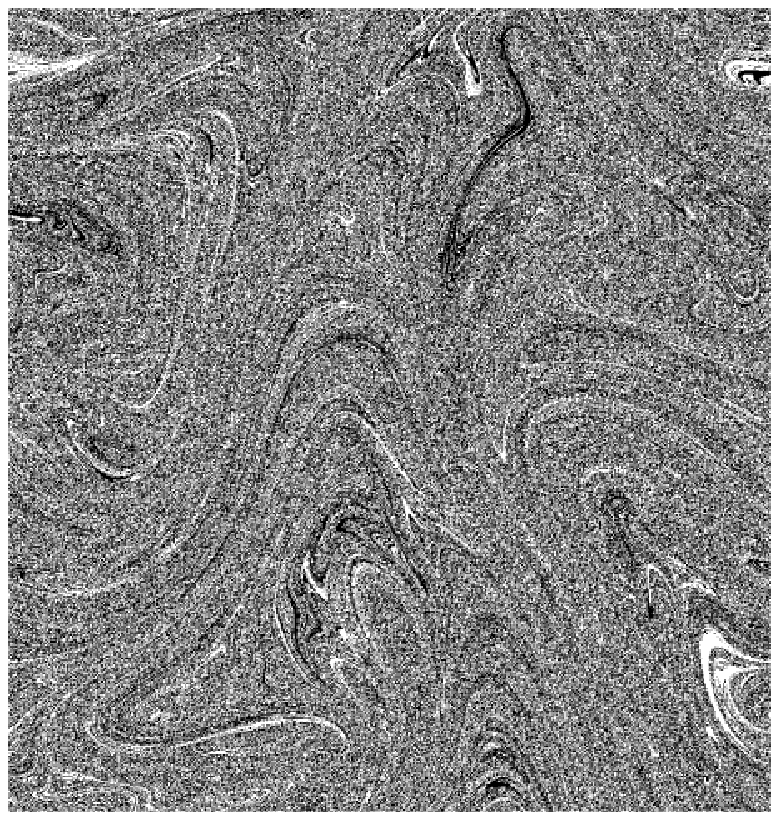}
\end{minipage}
\begin{minipage}{0.3\textwidth}
\includegraphics[scale=0.5]{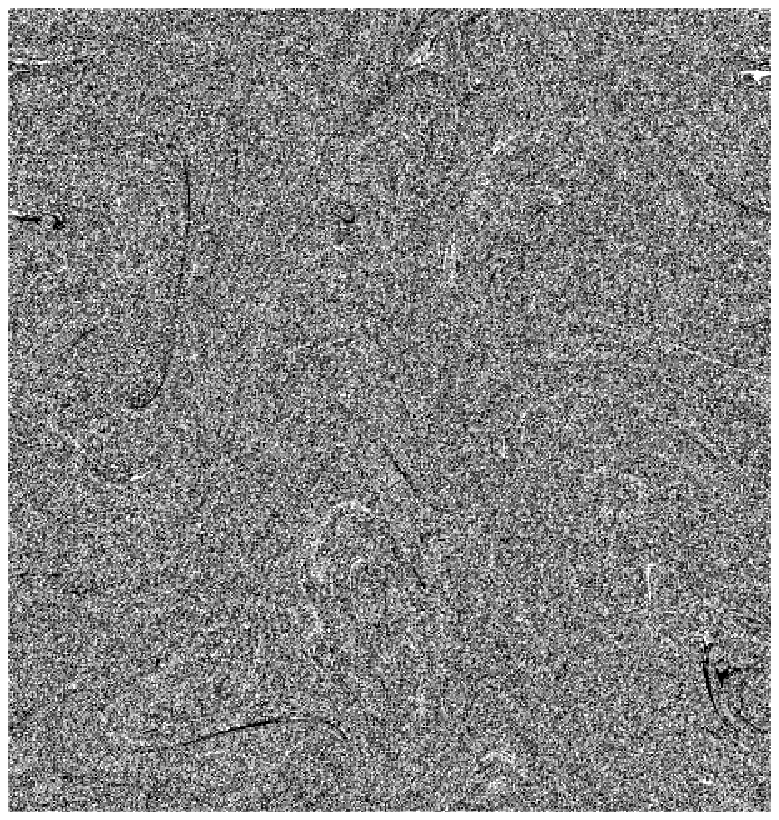}
\end{minipage}
\caption{Mixing in the Pierrehumbert model with $\tau=0.05$: snapshots at increasing times of the passive scalar concentration, starting from the initial datum on the top left-hand corner and moving clockwise.}
\end{figure}
The dynamics and mixing properties of $u$ are best understood at the Lagrangian level. 
Let $X(t, x) \in \T^2$ denote the position at time $t\geq 0$ of a passive tracer advected by $u$ and started at $ X(0,x)=x \in \T^2$. Then, it is well-known that $X(t, x)$ is a solution to the following ODE
\begin{align}
\partial_t X(t, x) = u(t, X(t, x)) \, , \quad X(0, x) = x \,. 
\end{align}
For $\beta\in [0,2\pi)$ and $x\in \T^2$, we define
\begin{align}\label{eq:PierreModel}
f^{H}_{\beta}(x) :=
\begin{pmatrix}
x_1+\tau\sin(x_2 -\beta)\\
x_2
\end{pmatrix},\qquad
f^{V}_{\beta}(x) :=  \begin{pmatrix}
x_1\\
x_2+\tau\sin(x_1 -\beta)
\end{pmatrix}.
\end{align}
Then, for a given tuple of random phases $\omega= (\omega^1,\omega^2)\in [0,2\pi)^2$, the  position $X(2\tau,x)\in \T^2$ of the particle at time $t=2\tau$ is given by $X(2\tau,x)=f_{\omega}(x)$, where
\begin{align}\label{eq:fomega}
f_{\omega}(x):=  f^{V}_{\omega^2} \circ f^{H}_{\omega^1} (x) \,. 
\end{align} 
Now, writing any sequence of possible random phase shifts as
\begin{align}
\underline{\omega}=(\omega_1,\omega_2,\ldots)\in \Omega:=\left([0,2 \pi)^{2}\right)^{\N},
\end{align}
we can recover the motion of tracer particle at times $2 n \tau$ by just iterating the maps given by $f_{\omega_i}$.
That is to say, $X(2\tau n, x) = f_{\underline \omega}^n(x)$ where 
\begin{align}\label{eq:randdyn}
f_{\underline \omega}^n(x):=f_{\omega_n}\circ \cdots \circ f_{\omega_1}(x).
\end{align}
One of the main results of this article is the following correlation decay estimate for the Pierrehumbert model.
\begin{thm} \label{thm:main}
Let $q,s>0$. There exists a function $D : \Omega \to [1,\infty)$ and  a constant $\alpha> 0$ such that for all mean-free
functions $\varphi, \psi\in H^s(\T^2) $, we have the almost sure estimate
\begin{align}\label{eq:quenc}
\left| \int_{\T^2} \varphi(x) \psi \circ f^n_\uo(x) \dd x \right| \leq D(\uo) \ee^{- \alpha n} \| \varphi\|_{H^s} \| \psi\|_{H^s} \, , 
\end{align}
while the function $D$ satisfies the moment bound $\E |D|^q < \infty$. 
\end{thm}
The exponential mixing property is encoded in the quenched correlation decay estimate \eqref{eq:quenc}. By duality, this implies the decay of the negative Sobolev norms \eqref{eq:negsob} of the solution $\rho$ to \eqref{eq:passivescal} with velocity field $u$ whose components are given by \eqref{eq:velbuild1}-\eqref{eq:velbuild2}.

\begin{rmk}
There are several ways to introduce randomness into this model. Besides the random phase shifting, one could also think of randomizing the switching time $\tau$ between the horizontal and vertical shearing. Although Theorem \ref{thm:main} is stated for the case of random phases and deterministic times, our approach can be 
used in the case of random phases and random times. The numerical evidence from \cite{CRFWZ21} suggests that randomizing phases and/or times is inessential,
and exponential mixing is observed even when no randomness is included in the model as long as $\tau$ is chosen to be sufficiently large. From a stochastic dynamics perspective, reducing randomness
limits the number of degrees of freedom available for establishing controllability and irreducibility of the various Markov processes associated to the Pierrehumbert system. Hence, the fully deterministic case
seems to be out of reach at the moment. An interesting open problem is the case in which the random phase sequence 
$\underline{\omega}$ is ``one-dimensional'', namely $\underline{\omega}=(\omega_1,\omega_2,\ldots )\in [0,2 \pi)^{\N}$. This amounts to modifying 
\eqref{eq:fomega} to $f_{\omega}= f^{2}_{\omega} \circ f^{1}_{\omega}$ for some random $\omega \in [0,2\pi)$.
\end{rmk}

\begin{rmk}
  While the velocity $u$ as in \eqref{eq:velbuild1} -- \eqref{eq:velbuild2} is not continuous in time, it is straightforward to build a continuous or even uniformly smooth-in-time version via an appropriate time-reparametrization, as described in \cite[Remark on p. 1914]{YZ17}. Therefore, Theorem \ref{thm:main} tell us that typical (i.e. almost all) realizations of the Pierrehumbert model are smooth-in-time, uniformly spatially $C^\infty$-in-time universal mixers. To our knowledge this is the first construction of such a mixer on the periodic box $\T^2$.
On the other hand, the mixer so-obtained is not time-periodic. Indeed, the ability to choose IID phase shifts is essential to our approach below.  Establishing universal mixing for the Pierrehumbert model using fixed, deterministic phase shifts is extremely difficult and is related to notorious open problems in smooth ergodic theory, see Section \ref{subsec:discussion} below. 
\end{rmk}


%
%
%
%
%
%


\subsection{A general dynamical system framework}
The Pierrehumbert model is just one instance of a general class of random dynamical systems  obtained as in \eqref{eq:randdyn}. 
Given a finite-dimensional compact Riemannian manifold without boundary  $M$, and a probability space of the form $(\Omega, \Fc, \P) = (\Omega_0, \Fc_0, \P_0)^{\N}$,  we consider the
random compositions of differentiable functions $f_{\omega}: M\to M$ with $\omega\in\Omega_0$, namely
\begin{align}
f^n_\uo = f_{\omega_n} \circ \cdots \circ f_{\omega_1}, \qquad n\in\N,
\end{align}
where $\uo=(\omega_1,\omega_2,\ldots )\in \Omega$. Typically in this paper we will assume $\Omega_0$ is a manifold, e.g., $\Omega_0 = \R^k, k \geq 1$, and that $\P_0 \ll $ Lebesgue measure. For the sake of this discussion, we will assume that almost-surely, $f_\omega : M \to M$ preserves Lebesgue measure $\pi$ on $M$. 

Markov chains arise from random dynamical systems by defining their transition kernels,
for any set $A \in {\rm Bor}(M)$, as
\begin{align}
P(x, A) = \P_0 ( f_\omega(x) \in A),\qquad  P^{n + 1}(x, A) = \int P^n(y, A) P(x, \dd y),
\end{align}
and the corresponding transition operator or Markov semigroup, acting on a continuous, bounded function $\varphi : M \to \R$, as
\begin{align}
P^n \varphi(x) = \int_M \varphi(y) P^n(x, \dd y).
\end{align}
There are several ingredients which are required for proving exponential mixing for random dynamical systems and they can be generically identified with non-degeneracy conditions on $f_\omega$ and $P$, and various processes associated to them. While we will not go discuss them in detail in this introduction, the following
two features are absolutely necessary.

\begin{itemize}
\item[(a)] {\bf The kernel $P$ is uniformly geometrically ergodic. } 
As $f_\omega$ preserves Lebesgue measure $\pi$ almost-surely, it follows that $\pi$ is stationary, i.e., 
\begin{align}
\pi(A) = \int P(x, A) \dd \pi(x), \qquad \forall A \in {\rm Bor}(M) \,. 
\end{align}
We say that $P$ is uniquely geometrically ergodic if $\pi$ is the unique stationary measure and is uniformly exponentially attracting, i.e., $P^n \varphi(x) \to \int \varphi \,\dd \pi $ as $n\to \infty$ exponentially fast for all $\varphi : M \to \R$ continuous and bounded. 

\item[(b)] {\bf Positivity of the top-Lyapunov exponent. }
This implies that almost everywhere in $M$ and with probability 1, nearby particles are pushed away at an exponentially fast rate. This is phrased in terms of the Jacobian $D_x f^n_\uo$ as the condition that the limit
\begin{align}\label{eq:defnLEIntro}
\lambda_1 := \lim_{n \to \infty} \frac1n \log | D_x f^n_\uo|
\end{align} 
exists and is greater than $ 0$ for $\P \times \pi$-a.e. $(\uo, x)$. We note here that the limit  \eqref{eq:defnLEIntro} exists and is almost-surely constant over $(\uo, x)$ under condition (a) (Theorem \ref{thm:MET} below). 
\end{itemize}
Note that (a) differs from the \emph{almost-sure} mixing we wish to prove: instead, (a) refers to mixing of the RDS averaged over all possible noise realizations, i.e. \emph{annealed mixing}.

Unfortunately, the conditions (a) and (b) mentioned above are far from being sufficient for exponential mixing: in particular, in the method pursued in this manuscript, we will also need unique geometric ergodicity 
for the so-called projective process
\begin{align}
\hat f_\omega(x, v) := \left( f_\omega (x), \frac{D_x f_\omega v}{|D_x f_\omega v|} \right), 
\end{align}
defined the sphere bundle of $M$,
and the two-point chain with transition kernel 
\begin{align}
P^{(2)} \left((x,y), K\right) = \P_0 \left( (f_\omega(x), f_\omega(y)) \in K\right),
\end{align}
defined on the product space $M \times M \setminus \Delta$, with $\Delta := \{ (x,x) : x \in M\}$. Precise definitions of these processes and the conditions required on them are stated in Section \ref{sec:mixing}. The consequence of the above discussion is the following theorem, which we state here informally.

\begin{thm}\label{thm:main2}
Assume that the dynamical system generated by $f_\omega$ is uniquely and uniformly geometrically ergodic, 
along with the corresponding projective and two-point processes, and possesses a positive Lyapunov exponent.  
For any $q,s>0$, there exists a function $D : \Omega \to [1,\infty)$ and  a constant $\alpha> 0$ such that for all mean-free
functions $\varphi, \psi\in H^s(M) $, we have the almost sure estimate
\begin{align}
\left| \int_{\T^2} \varphi(x) \psi \circ f^n_\uo(x) \dd x \right| \leq D(\uo) \ee^{- \alpha n} \| \varphi\|_{H^s} \| \psi\|_{H^s} \, , 
\end{align}
while the function $D$ satisfies the moment bound $\E |D|^q < \infty$. 
\end{thm}

For a summary of the precise conditions under which Theorem \ref{thm:main2} holds, 
see Section \ref{subsec:quenchedMixing}.

\subsection{Discussion} \label{subsec:discussion}


 For a volume-preserving dynamical system, random or deterministic, a positive Lyapunov exponent as in equation \eqref{eq:defnLEIntro} implies that time-$n$ linearizations exhibit exponential stretching and contracting in tangent space. In the dynamics community, this stretching/contracting is called hyperbolicity, and has long been understood as the primary mechanism responsible for rapid mixing. 
Systems for which the realization of hyperbolicity is uniform over phase space are called \emph{Anosov maps} or \emph{flows}, and by now there are well-known, checkable conditions for showing that Anosov systems are universal exponential mixers: for discrete-time maps, see, e.g., \cite{bowen2008equilibrium, baladi2000positive, liverani1995decay}, and for the (considerably more subtle) case of continuous-time flows, see, e.g., \cite{dolgopyat1998decay, liverani2004contact}. Exponential mixing for hyperbolic systems is a vast subject to which we cannot do justice here. 

In applications, the Anosov property is often far too restrictive: for instance, there are strong limitations on which phase spaces admit Anosov maps or flows \cite{gromov1981groups, smale1967differentiable}, and those that exist are subject
to strong limitations \cite{manning1974there}. 
It is natural instead to study \emph{nonuniformly hyperbolic systems}, i.e., those for which the realization of stretching/contracting is nonuniform over phase space; see, e.g., the survey \cite{young2013mathematical}. This motivates the use of the Lyapunov exponent as in \eqref{eq:defnLEIntro}: a system is called \emph{nonuniformly hyperbolic} if the value \eqref{eq:defnLEIntro} exists and is strictly positive at a positive-volume subset of initial $x \in M$. While existence of the limit follows from standard ergodic theory tools (Theorem \ref{thm:MET} below), the limit need not be uniform over phase space, hence the term ``nonuniform''. 
Nonuniform hyperbolicity is expected to be the mechanism behind typical incompressible mixers \cite{pesin2010open}, e.g., via the much-discussed stretch-and-fold mechanism \cite{childress1995stretch}. 

We note that there exist several sets of theoretical tools for establishing 
exponential mixing for nonuniformly hyperbolic systems, e.g., the framework of 
Young towers \cite{young1998statistical, young1999recurrence}. Unfortunately, checking the conditions of these frameworks for systems of practical interest is notoriously difficult~\cite{pesin2010open}. For instance, it is an outstanding open question \cite{crovisier2019problem} to verify nonuniform hyperbolicity for the Chirikov standard map \cite{chirikov1979universal}, a deterministic, discrete-time toy model of the stretch-and-fold mechanism in Lagrangian flow \cite{crisanti1991lagrangian}. 

Establishing nonuniform hyperbolicity is far more tractable in the presence of nondegenerate noise -- see, e.g., \cite{blumenthal2017lyapunov} and the discussion and references within. In our setting, we take advantage of Furstenberg's criterion \cite{furstenberg1963noncommuting}, which in its original form gives checkable conditions for the Lyapunov exponent of an IID product of determinant 1 matrices to be strictly positive. In our setting, an appropriate extension of Furstenberg's criterion implies that under a suitable set of nondegeneracy conditions on the law of a volume-preserving RDS $f_\omega, \omega \sim \P_0$, the Lyapunov exponent is positive (Section \ref{subsec:furstenberg3}). There is a long literature of results \`a l\`a Furstenberg'', e.g., \cite{royer1980croissance, ledrappier1986positivity, carverhill1987furstenberg, virtser1980products, bougerol2012products, avila2010extremal}. In this manuscript, we owe much to the approach of \cite{bougerol1988comparaison} in particular.

To go from nonuniform hyperbolicity to correlation decay: the method we present here first appears in \cite{dolgopyat2004sample} (see also \cite{ayyer2007quenched}), which establishes that with probability 1, the sequences $\varphi \circ f^n_\uo$ satisfy a CLT when the $f^n_\uo$ come from the time-$n$ maps of a divergence-free stochastic differential equation satisfying some hypoellipticity conditions. The approach in \cite{dolgopyat2004sample} is based off of a large deviations theory for the two point Markov chain with kernel $P^{(2)}$ defined above, 
derived in \cite{baxendale1989lyapunov}. 

A version of the methods of this manuscript was employed in \cite{bedrossian2018lagrangian, bedrossian2019almost} to establish almost-sure exponential mixing for the Lagrangian flow of stochastic Navier--Stokes on the periodic box with nondegenerate noise. Although many of the basic ideas are the same, there are several significant differences
between this and the approach of this manuscript, the most significant being that  \cite{bedrossian2018lagrangian, bedrossian2019almost} is geared towards continuous time systems, while in the present manuscript we work with discrete time systems, entailing several technical problems to be overcome (see, e.g., Remark \ref{rmk:strongFeller}). 

\subsection*{Organization of the paper}

We start in Section \ref{sec:rdsPrelims} with some background and preliminary results for Markov chains and random dynamical systems in an abstract setting, e.g., Harris' Theorem for verifying geometric ergodicity and Furstenberg's criterion for establishing positivity of Lyapunov exponents. Section \ref{sec:LE} specializes to random dynamical systems subjected to absolutely continuous noise, establishing tools for checking the hypotheses of Harris' Theorem and Furstenberg's criterion in this setting. Section \ref{sec:mixing} completes the proof of universal mixing in the abstract setting, and Section \ref{sec:pierrehumbert} applies the abstract result to the Pierrehumbert model.

\section{Markov chains and random dynamical systems}\label{sec:rdsPrelims}

Some of the concepts used in the sequel are best introduced in a general abstract setting. To fix notation, in what follows $X$ denotes a complete metric
space not necessarily compact, and ${\rm Bor}(X)$ is its Borel $\sigma$-algebra. Let $\Mc(X)$ denote the space of probability measures on $X$, endowed with either the weak$^*$ topology
or the total variation metric
\begin{equation}
d_{TV}(\mu, \nu) := \sup_{A \in {\rm Bor}(X)}|\mu(A) - \nu(A)| \,. 
\end{equation}
We begin by covering some preliminaries on random dynamical systems in an abstract setting. 

\subsection{Markov chain preliminaries}
Assume that $P$ is a Markov transition kernel on $X$, namely,  for each $x \in X$ we have that 
$P(x, \cdot)$ is a Borel probability on $X$. For a $n\in \N$, iterates of $P$ are defined inductively by the Chapman--Kolmogorov relation
\begin{equation}
P^{n + 1}(x, A) = \int P(x, \dd y) P^n(y, A) \, .
\end{equation}
Moreover, we will assume that $P$ has the \emph{Feller} property, i.e. if $\varphi : X \to \R$ is a bounded, continuous function on $X$, then 
\begin{equation}
x \mapsto P \varphi(x) := \int P(x, \dd y) \varphi(y)
\end{equation}
 is continuous as well. That is to say if $x \mapsto P(x, \cdot)$ is continuous in the weak$^*$ topology on $\mathcal M(X)$, then $P$ has the Feller property. We will also at times refer to the dual action $P^*$
 on probability measures: given $\mu \in \mathcal M(X)$, we define
\begin{equation}
 P^* \mu(A) := \int P(x, A) \dd \mu(x), \quad \text{ for } A \in {\rm Bor}(X) \,. 
\end{equation}

A Borel probability measure $\pi \in \mathcal M(X)$ is called 
a {\it stationary} for $P$ if it is a fixed-point of $P^*$, i.e., 
\begin{equation}
\pi(A) = \int P(x, A) \dd \pi(x), \qquad \forall A\in  {\rm Bor}(X).
\end{equation}
A set $A\in  {\rm Bor}(X)$ is  $(P, \pi)$-invariant if $P \chi_A = \chi_A$ holds $\pi$-almost everywhere, 
where $\chi_A$ is the indicator function of $A$. 
We say that $\pi$ is an {\it ergodic stationary measure} if all $(P, \pi)$-invariant sets have 
$\pi$-measure zero or one. 
In the setting described above, if $\pi$ is the unique stationary measure of a transition kernel $P$ 
then it is automatically ergodic -- this is a standard fact and follows from, for example,~\cite[Proposition I.2.1]{kifer2012ergodic}.

The following properties of Markov chains are commonly used in the arguments involving mixing in the subsequent sections. 

\begin{defn}\label{defn:smallSet}
Let $n \geq 1$.  We say that a set $A \subset X$ is $P^n$-\emph{small}  if there exists a positive measure $\nu_n$ on $X$ such that, for all $x \in A$, we have that
\[
P^n(x, B) \geq \nu_n(B) \quad \text{ for all Borel } B \subset X\, . 
\]
We say that $A\subset X$ is small if it is $P^n$-small for some $n\geq 1$.
\end{defn}

Lastly, we recall the following drift-type condition. 
\begin{defn}\label{defn:driftCond}
We say that a function $V : X \to [1,\infty)$ satisfies a {\it Lyapunov--Foster drift condition} 
if there exist $ \alpha \in (0,1), b > 0$ and a compact set $C \subset X$ such that
\begin{equation}
P V \leq \alpha V + b \chi_C
\end{equation}
Given such a $V$, we define the weighted norm
\begin{equation}
\| \varphi\|_{V} := \sup_{x \in X} \frac{|\varphi(x)|}{V(x)} ,
\end{equation}
and define $M_V(X)$ to be the space of measurable observables $\varphi : X \to \R$ 
such that $\| \varphi \|_{V} < \infty$.
\end{defn}

\subsection{An abstract Harris theorem}

For irreducible, aperiodic finite-state Markov chains, the Perron--Frobenius theorem asserts the unique existence of a stationary probability and a quantitative geometric rate of convergence to that stationary measure. 
An analogue for Markov chains on a complete metric space is Harris's theorem,  formulated below in a version suitable for our purposes.

\begin{thm}[Abstract Harris Theorem, \cite{meyn2012markov}]\label{thm:Harris}
Let $P$ be a Feller transition kernel and assume the following: 
\begin{enumerate} [label=(\alph*), ref=(\alph*)]
\item\label{item:smallset} (Small sets) There exists an \emph{open} small set. 
\item\label{item:irred} (Topological irreducibility) For every $x \in X$ and nonempty open set $U \subset X$, there exists $N = N(x, U) \geq 1$ such that 
$P^N(x, U) > 0$.
\item\label{item:apero} (Strong aperiodicity) There exists $x_* \in X$ such that for all open $U \ni x_*$, we have that $P(x_*, U) > 0$; 
\item \label{item:drift} (Drift condition) There exists a function $V : X \to [1,\infty)$ satisfying a Lyapunov--Foster drift condition for $P$. 
\end{enumerate}
Then, $P$ is $V$-uniformly geometrically ergodic, i.e. $P$ admits a unique stationary measure $\pi$, and
has the property that there exists $D > 0$ and $\gamma \in (0,1)$ such that
 for all $x \in X$ and $\varphi \in M_V(X)$, we have
\begin{equation}\label{eq:geomerg}
\left|  P^n \varphi(x) - \int \varphi \,\dd \pi \right| \leq D V(x) \| \varphi\|_{V} \gamma^n \,. 
\end{equation}
\end{thm}
The geometric ergodicity statement \eqref{eq:geomerg} can be equivalently reformulated
in terms of the operator norm of $P^n$, the conclusion being that $\| P^n - \pi\|_V \leq D \gamma^n$, where here we abuse notation and identify $\pi$ with the operator $\varphi \mapsto \int \varphi d \pi$. 
%
While this result is essentially borrowed\footnote{While for the sake of consistency we have attempted to use the same terminology as \cite{meyn2012markov}, we caution that our definition of ``strong aperiodicity'' given in Theorem \ref{thm:Harris} above is not the same as that in \cite{meyn2012markov}. See Appendix \ref{app:harris} for details.} from \cite{meyn2012markov}, it is 
not explicitly stated there, and is instead an assembly of several results
found throughout that book along with some additional arguments. For the sake of completeness, a proof is sketched
in Appendix \ref{app:harris}. 

In the special case when $X$ is compact, item \ref{item:drift} in Theorem \ref{thm:Harris} is redundant, as
 a Markov kernel $P$ automatically satisfies a Lyapunov--Foster drift condition for the drift function 
 $V \equiv 1$. In this case, we say that the Markov kernel $P$ is {\it uniformly geometrically ergodic}.

\subsection{Random dynamical systems}

In this manuscript, we will exclusively work with Markov chains derived from 
\emph{random dynamical systems} (RDS) with independent increments, 
a framework which we now describe. Let $(\Omega_0, \Fc_0, \P_0)$ be a fixed probability space. 
A continuous RDS with independent increments on $X$ is an assignment to each $\omega \in \Omega_0$
of a continuous mapping $f_\omega : X \to X$ which satisfies the following 
the following mild
measurability condition: for all $A \subset {\rm Bor}(X)$ and $x \in X$, the set
$\{ \omega \in \Omega_0 : f_\omega(x) \in A\}$ is $\Fc_0$-measurable. 
Defining $(\Omega, \Fc, \P) = (\Omega_0, \Fc_0, \P)^{\N}$, with elements
$\uo \in \Omega$ written $\uo = (\omega_1, \omega_2, \cdots)$, we consider the 
random compositions of functions
\begin{equation}
f^n_\uo = f_{\omega_n} \circ \cdots \circ f_{\omega_1}, \qquad n\in\N,
\end{equation}
following the convention that $f^0_\uo$ is the identity mapping on $X$. 
As customary,  $\theta : \Omega \to \Omega$ indicates the leftward shift on $\Omega$, i.e.,
if $\uo = (\omega_1, \omega_2, \cdots)$ then $\theta \uo = (\omega_2, \omega_3, \cdots)$. 
Note that $\theta$ is a measure-preserving transformation on $(\Omega, \Fc, \P)$, i.e., 
$\theta^{-1} \Fc \subset \Fc$ and $\P \circ \theta^{-1} = \P$, and moreover $f^n_\uo$ satisfies
the  \emph{cocycle property}: 
\begin{equation}
f^{n + m}_\uo = f_{\theta^m \uo}^n \circ f_\uo^m, \qquad \forall m, n \geq 0 .
\end{equation}
%
%
Continuous RDS in this framework naturally give rise to Markov chains with kernels
\begin{equation}
P(x, A) = \P_0 ( f_\omega(x) \in A) \, .
\end{equation}
Continuity of the $f_\omega$ implies automatically that the transition kernel $P$ is Feller. 
We also note that stationarity of $\pi \in \mathcal M(X)$ can be rewritten as
\begin{equation}
\pi(A) = \int  \pi \circ f_\omega^{-1}(A) \,  \dd \P_0(\omega),  \qquad \forall A \in {\rm Bor}(X) \,. 
\end{equation}
Informally, $\pi$ is ``$f_\omega$-invariant on average''. We will abuse terminology in what follows and describe a measure $\pi$ as a
\emph{stationary measure for an RDS} $f^n_\uo$ if it is a stationary measure for its corresponding Markov kernel $P$.

%


\begin{defn}
We say that a bounded, measurable function $\varphi : X \to \R$ is $(P, \pi)$\emph{-invariant}
if $P \varphi = \varphi$ holds $\pi$-a.e. We say that $\pi$ is \emph{ergodic} if all $(P, \pi)$-invariant functions are $\pi$-almost surely constant. 
\end{defn}

%

\subsubsection{Linear cocycles over a random dynamical system}

Let $d \geq 2$ and let $GL_d(\R)$ denote the space of invertible $d \times d$-matrices with real entries. 

\begin{defn}
Let $\Ac : \Omega_0 \times X \to GL_d(\R)$ be a measurable mapping. The {\it linear cocycle} generated by $\Ac$ is the composition
\begin{equation}
\Ac^n_{\uo, x} := \Ac_{\omega_{n}, f^{n-1}_\uo(x)} \circ \cdots  \circ \Ac_{\omega_1, x} ~ \in GL_d(\R) 
\end{equation}
for $n \geq 1$. We call $\Ac$ a {\it continuous linear cocycle} if $x \mapsto \Ac_{\omega, x}$ is continuous for $\P_0$-almost every $\omega \in \Omega_0$. 
\end{defn}
The following version of the multiplicative ergodic theorem (MET) stated below describes the asymptotic behavior 
of linear cocycles in this setting. 

\begin{thm}[{\cite[Theorem III.1.1]{kifer2012ergodic}}]\label{thm:MET}
Let $f^n_\uo$ be a continuous RDS with ergodic stationary measure $\pi \in \mathcal M(X)$. 
Assume $\Ac$ satisfies the integrability condition
\begin{equation}
\int \left(  \log^+|\Ac_{\omega, x}| + \log^+ | \Ac_{\omega, x}^{-1}| \right) \, \dd \P_0(\omega) \dd \pi(x) < \infty \,. 
\end{equation}
Then, there exist $r \in \{1, \cdots, d\}$, constant values
 $\lambda_1 > \cdots > \lambda_{r - 1} > \lambda_r > -\infty$, and, at $\P \times X$-almost every $(\uo, x) \in \Omega \times X$, a filtration
\begin{equation}
 \R^d =: F_1(\uo, x) \supsetneq F_2(\uo, x) \supsetneq \cdots \supsetneq F_r(\uo, x) \supsetneq F_{r + 1}(\uo, x) := \{ 0 \} 
\end{equation}
 of subspaces with the property that 
\begin{equation}
 \lambda_i = \lim_{n \to \infty} \frac{1}{n} \log | \Ac^n_{\uo, x} v | 
\end{equation}
 for all $v \in F_i(\uo, x) \setminus F_{i + 1}(\uo, x)$. Moreover, the 
 assignments $(\uo, x) \mapsto F_i(\uo, x)$ vary measurably for all $i$, and $\dim F_i(\uo, x)$ is constant over $\P \times \pi$-typical $(\uo, x) \in \Omega \times X$. 
\end{thm}

The values $\lambda_i$ are called {\it Lyapunov exponents} with {\it multiplicities} 
\begin{equation}
m_i := \dim F_i(\uo, x) - \dim F_{i + 1}(\uo, x) \, ,
\end{equation}
while the filtration $F_i$ is referred to as the {\it Oseledets splitting}. 

\begin{rmk}\label{rmk:sumLE3}
The proof of Theorem \ref{thm:MET} due to Ragunathan \cite{raghunathan1979proof} realizes the Lyapunov exponents as the distinct values among the values
\begin{equation}
	\chi_i := \lim_{n \to \infty} \frac1n \log \sigma_i(\Ac^n_{\uo, x}) \, . 
\end{equation}
Here, $\sigma_i$ is the $i$-th singular value of a matrix. 
The RHS limit exists and is constant (in particular deterministic/non-random) for $\P \times \pi$-a.e. $(\uo, x)$ 
by an argument using the 
subadditive ergodic theorem \cite{kingman1973subadditive}. In turn, the weights $m_i$ are given by
\begin{equation}
	m_i = \# \{ 1 \leq j \leq d : \chi_j = \lambda_i \} \,. 
\end{equation}
In view of the identity $|\det(B)| = \prod_{i =1 }^d \sigma_i(B)$ for $d\times d$ matrices $B$,  the weighted sum
\begin{equation}
\lambda_\Sigma := \sum_{i = 1}^r m_i \lambda_i \in [-\infty, \infty)
\end{equation}
is given by the limiting formula
\begin{equation}
\lambda_\Sigma = \lim_{n \to \infty} \frac1n \log | \det (\Ac^n_{\uo, x})|
\end{equation}
for $\P \times \pi$-almost every $(\uo, x)$. 
In particular, if $\det(\Ac_{\omega, x}) \equiv 1$, as is the case in our applications, then $\lambda_\Sigma = 0$ automatically. 
\end{rmk}

Naturally associated to Lyapunov exponents is the {\it projectivized dynamics} of 
the cocycle on tangent directions: given a linear cocycle over an RDS $f^n_\uo$ 
generated by $\Ac$, we define the {\it projective RDS} $\widehat f^n_\uo$ on $X \times S^{d-1}$ by
\begin{equation}
\widehat f_\omega (x, v) := \left(f_\omega(x), \widehat \Ac_{\omega, x}(v)  \right), \qquad 
\widehat f^n_\uo = \widehat f_{\omega_{n}} \circ \cdots \circ \widehat f_{\omega_1} \, . 
\end{equation}
Here and throughout, given an invertible $d \times d$ matrix $B$ and $v \in S^{d-1}$ we write $\widehat B : S^{d-1} \to S^{d-1}$ for the mapping
\begin{equation}
\widehat B(v) = \frac{B v}{|B v|} \,. 
\end{equation}
As a continuous RDS, $\widehat f^n_\uo$ gives rise to a corresponding Markov kernel $\widehat P$ on 
$X \times S^{d-1}$. 
Stationary measures $\nu \in \mathcal M(X \times S^{d-1})$ 
of $\widehat P$ are referred to as {\it projective stationary measures}, 
and describe the equilibrium statistics of tangent directions $\Ac^n_{\uo, x} v$ for $v \in \R^d$. 
Note that any projective stationary measure $\nu$ projects to a stationary measure for $P$, i.e. 
the measure $\pi(K) := \nu(K \times S^{d-1})$ is a stationary measure for $P$. 

At times we will refer to {\it measurable families of measures} $(\nu_x)_{x \in S}$ for some measurable $S \subset X$. These are collections of measures $\nu_x$ on $S^{d-1}$ with the property that $x \mapsto \nu_x(K)$ varies measurably in $x$ for all fixed Borel $K \subset S^{d-1}$. We say that $(\nu_x)_{x \in S}$ is weak$^*$ continuous if $x \mapsto \int \psi(v) \dd \nu_x(v)$ varies continuously in $x$ for all fixed continuous $\psi : S^{d-1} \to \R$. 


\subsubsection{Furstenberg's criterion}

Assume $\Ac^n_{\uo, x}$ is a linear cocycle over an RDS $f^n_\uo$ on a complete metric space $X$ and that the assumptions of Theorem \ref{thm:MET} hold true. It is not hard to check that, without further conditions, we have 
\begin{align}
d \lambda_1 \geq \lambda_\Sigma \,. 
\end{align}
Equality in this bound, namely $d \lambda_1 = \lambda_\Sigma$, would imply $r = 1$, i.e., 
there is a sole Lyapunov exponent with multiplicity $d$. Furstenberg's criterion, originally obtained for the Lyapunov exponents of i.i.d matrices  \cite{furstenberg1963noncommuting}, provides a way of ruling out this degenerate scenario. Since then, this circle of ideas has been extended to a variety of settings -- see  \cite{ledrappier1986positivity, avila2010extremal} and the references therein. The following version can be derived from \cite[Proposition 2, Theorem 3]{ledrappier1986positivity}; see also  \cite[Theorem 2.4]{royer1980croissance}. 

\begin{thm} \label{thm:Furstenberg}
Under the assumptions of Theorem \ref{thm:MET}, if $d \lambda_1 = \lambda_\Sigma$, then there is a measurable family $(\nu_x)_{x \in \Supp(\pi)}$ satisfying the following property: for $\P \times \pi$-almost every $(\uo, x)$ and all $n \geq 1$ 
\begin{align} \label{eq:furstDegenTimeN2}
(\Ac_{\uo, x}^n)_* \nu_x = \nu_{f^n_\uo x} \, .
\end{align}\end{thm}

Above, we write
\begin{align}\label{eq:pushForProject}
B_* \nu_x := \nu_x \circ \widehat B^{-1}
\end{align}
 for the {\it pushforward} of $\nu_x$ by some measurable $\widehat B : S^{d-1} \to S^{d-1}$. 
 

 We emphasize that \eqref{eq:furstDegenTimeN2} is a degenerate scenario, which can be ruled out if the law of $(f_\uo^n x, \Ac_{\uo, x}^n)$ in $X \times GL_d(\R)$ has ``broad enough support''. To see this on a formal level, 
fix $y  \in X$  
and consider the space of sample paths $\uo \in \Omega$ such that $y = f^n_\uo x$. 
 Observe that the matrix $\Ac^n_{\uo, x} $ depends explicitly on the whole sample path $\uo$, while $\nu_y = \nu_{f^n_\uo x}$ depends only on $\uo$ implicitly through the fixed value of $y = f^n_\uo x$. 
 From here, it follows that if \eqref{eq:furstDegenTimeN2} holds, then the law of $\Ac^n_{\uo, x}$ conditioned on $y = f^n_\uo x$ must have empty interior in $GL_d(\R)$ (c.f. Lemma \ref{lem:emptyInterior} below). 
 We conclude that if the law of $\Ac^n_{\uo, x}$ on $GL_d(\R)$ conditioned on the event $f^n_\uo x = y$ has nonempty interior, then \eqref{eq:furstDegenTimeN2} cannot occur, and so $d \lambda_1 > \lambda_\Sigma$ must hold. 
 However, there are some technical hurdles in the above argument: (i) the law of $\Ac^n_{\uo, x}$ conditioned on $y = f^n_\uo x$ can be difficult to work with directly, and (ii) the family $x \mapsto \nu_x$ varies measurably. 
 
The following work-around is inspired by Bougerol \cite{bougerol1988comparaison}, and allows us to work instead with a weak$^*$ continuous family $\{\nu_x\}$. 

\begin{prop}\label{prop:bougerolRegularity}
Assume that $X$ is compact and that $P$ is uniformly geometrically ergodic.
%
%
If $d \lambda_1 = \lambda_\Sigma$, then there is a family $(\nu_x)_{x \in \Supp(\pi)}$ of probability measures on $S^{d-1}$ such that 
the mapping $x \mapsto \nu_x$ varies weak$^*$ continuously on $\Supp(\pi)$,  
and so that 
\begin{align}
\nu_x = (\Ac^n_{\uo, x})^T_* \nu_{f^n_\uo x}
\end{align}
for $\P \times \pi$-almost every $(\uo, x)$ and all $n \geq 1$. 
\end{prop}

Compared to \eqref{eq:furstDegenTimeN2}, $\Ac^n_{\uo, x}$ appears on the right-hand side, applied to $\nu_{f^n_\uo x}$, and under a transpose.

\begin{cor}\label{cor:bougerol}
Assume the setting of Proposition \ref{prop:bougerolRegularity}. 
If $d \lambda_1 = \lambda_\Sigma$, then for all $n \geq 1$ we have 
\begin{align}
\nu_x = (A^{T})_* \nu_{y}
\end{align}
holds for \emph{all} $x \in \Supp(\pi)$ and all $(y, A)$ in the support
of $\operatorname{Law}(f_\uo^n x, \Ac^n_{\uo, x})$ on $X \times GL_d(\R)$. 
\end{cor}

While Proposition \ref{prop:bougerolRegularity} and Corollary \ref{cor:bougerol} 
follow from arguments in Bougerol's original paper \cite{bougerol1988comparaison}, 
 the results are not stated in the above form. For the sake of completeness, sketches of the proofs of Proposition \ref{prop:bougerolRegularity} and Corollary \ref{cor:bougerol} are given in Appendix \ref{app:bougerol}.

\begin{rmk}
Recall that a Markov kernel is called \emph{strong Feller} if $\varphi : X \to \R$ bounded measurable implies $P \varphi : X \to \R$ is continuous. 
There are several works on the application of Furstenberg-type ideas
to the Lypaunov exponents of RDS, e.g., to continuous-time RDS satisfying the strong Feller property \cite{baxendale1989lyapunov, bedrossian2018lagrangian}, dealing in particular with the issue of obtaining some regularity of the measurable family $x \mapsto \nu_x$. 
The version chosen here is particularly well-suited to continuous cocycles over RDS 
which are uniformly geometrically ergodic but lack the strong Feller property. 
\end{rmk}



\section{Random dynamical systems with absolutely continuous noise}\label{sec:LE}

In this section, we begin to specialize to the class of random 
dynamical systems pertaining to our main results, i.e those posed on 
finite-dimensional compact manifolds and forced 
by noise drawn from a smooth manifold and absolutely continuous with respect to the Lebesgue measure. 
We  focus on obtaining checkable sufficient conditions
for several of the properties we will use later on, namely, the 
existence of open small sets (Proposition \ref{prop:Tchain3}), strong aperiodicity (Lemma \ref{lem:aperiodicSuffCond}) and for ruling out Furstenberg's criterion (Proposition \ref{prop:suffCondRuleOutFurst}). 
%



In what follows, we will consider classes of continuous RDS
$f_\omega$ acting on a compact manifold $M$ without boundary. 
Throughout this section, we assume the following:

\begin{enumerate} [label=(R), ref=(R)]
	\item\label{ass:reg} We have that $\Omega_0$ is a smooth, complete Riemannian manifold. The law $\P_0$ on $\Omega_0$ admits a density $\rho_0$ 
	with respect to Lebesgue measure $\dd \omega$ on $\Omega_0$. 
	Additionally, the mapping $\Omega_0 \times M \to M$
	given by 
\begin{equation}
	(\omega, x) \mapsto f_\omega (x) 
\end{equation}
	is $C^2$. 
\end{enumerate}


Several of the conditions we pose involve viewing 
the RDS $f_\omega$ as a function of finitely many of the coordinates
$\omega_i$ of the sequence $\uo = \{\omega_i\}_{i \geq 1}$. 
To this end, we adopt the notation 
$\uo^{n} = (\omega_1, \cdots, \omega_n)$ for
elements of the space $\Omega_0^n$, the Cartesian product of $n$ copies of $\Omega_0$. 
With a slight abuse of notation, we will use the symbol $\rho_0$ for the density of the product law $\P_0^n$ on $\Omega_0^n$, defined as
\begin{equation}
\rho_0(\uo^{n}) = \rho_0(\omega_1) \cdots \rho_0(\omega_n).
\end{equation}



\subsection{Sufficient conditions for small sets and aperiodicity}

First, we give a checkable sufficient condition for 
RDS satisfying assumption \ref{ass:reg} to admit open small sets (Definition \ref{defn:smallSet}).

Our condition is framed, for fixed $x \in M$, in terms of the mapping $\Phi_x : \Omega_0^n \to M$ defined for $\uo^{n} = (\omega_1, \cdots, \omega_n)$ by
\begin{equation}
\Phi_x(\uo^{n}) := f_{\omega_n} \circ \cdots \circ f_{\omega_1}(x) \,. 
\end{equation}

\begin{prop}\label{prop:Tchain3}
Assume the RDS $f_\omega$ satisfies assumption  \ref{ass:reg}, and there exist $n \geq 1$ and 
$(\uo_\star^{n}, x_\star) \in \Supp(\rho_0) \times \Supp(\pi)$ which satisfy the following properties:
\begin{enumerate} [label=(\roman*), ref=(\roman*)]
	\item There exist some $c > 0, \eps > 0$ such that $\rho_0(\uo^{n}) \geq c > 0$ if $|\uo^{n} - \uo^{n}_\star| < \eps$. \label{prop:Tchain3:i}
	\item $\Phi_{x_\star}$ is a submersion at $\uo^{n} = \uo^{n}_\star$. \label{prop:Tchain3:ii}
\end{enumerate}
Then, $P$ admits an open, $P^n$-small set with corresponding measure $\nu_n \ll \Leb_M$. 
\end{prop}

\begin{proof}
It suffices to consider the case when $n = 1$; the general proof is similar. By the constant rank theorem, submersion
 implies there are $C^2$ coordinate charts near $\omega_\star := \uo^{1}_\star$ in $\Omega_0$
  and near $f_{\omega_\star} (x_\star)$ in $M$ 
  for which $\omega_\star \mapsto \Phi_{x_\star}(\omega_\star)$ is an orthogonal projection. It is immediate that the pushforward $(\Phi_{x_\star})_* \P_0$ 
 is absolutely continuous with respect to the Lebesgue measure $\Leb$ on $M$ and has density bounded below by a constant $c'>0$ in an $\eps'$-neighborhood of $f_{\omega_\star} (x_\star)$ for some $\eps' > 0$. 
  
Since having full rank is an open property in the space of matrices, it follows that $\Phi_y$ is a submersion for all $y$ in a small neighborhood of $x_\star$, and the same argument yields that $(\Phi_y)_* \P_0$  is absolutely continuous with respect to $\Leb$ on $M$ with density bounded below by $c'/2$ in a small neighborhood of $f_{\omega_\star}(x_\star)$.
Note we have used here the fact that for $r \geq 1$, the $C^r$ data of the charts supplied by the constant rank theorem are controlled by the $C^r$ data of $(\omega, x) \mapsto f_\omega(x)$. 

All in all, we conclude 
\begin{equation}
P(x, K) \geq \frac{c'}{2} \Leb(K \cap U') 
\end{equation}
 for all $x$ in an open neighborhood $V$ of $x_\star$, where $U'$ is a small open neighborhood of $\Phi_{x_\star}(\omega_\star)$. We conclude $V$ is a $\nu$-small set with $\nu(K) := \frac{c'}{2} \Leb(K \cap U')$. 
\end{proof}

We now give a simple condition for the Markov kernel $P$ to be strongly aperiodic in the sense of Theorem \ref{thm:Harris}. 

\begin{lem}\label{lem:aperiodicSuffCond}
Assume there exist $\omega_\star \in \Supp(\P_0)$ and $ x_\star \in M$ such that 
$f_{ \omega_\star}( x_\star) =  x_\star$. Then, for any open $U \subset M$ containing $x_\star$ we have that $P(x_\star, U) > 0$. 
\end{lem}
The proof follows from the continuity of $\omega \mapsto f_\omega(x)$ and is omitted for brevity. 

\subsection{Ruling out Furstenberg's criterion}\label{subsec:furstenberg3}

We now turn to the treatment of the derivatives $D_x f^n_\uo$ 
as a linear cocycle over  $f_\omega$, and will provide
sufficient conditions under which the degeneracy $d \lambda_1 = \lambda_\Sigma$
can be ruled out using Furstenberg's criterion (Proposition \ref{prop:bougerolRegularity}). 
To start, the following condition ensures that Lyapunov exponents exist 
as in Theorem \ref{thm:MET}. 

\begin{enumerate} [label=(I), ref=(I)]
	\item\label{ass:logbdd} 
	For a.e. $\omega$, we have that $f_\omega$ 
	is a local diffeomorphism, and
	 admits an ergodic stationary measure $\pi$ on $M$ 
	 with the property that
	\begin{align}
	\int \left(  \log^+ | D_x f_\omega| + \log^+|(D_x f_\omega)^{-1}| \right) \dd \pi(x)\dd \P_0(\omega) < \infty
	\end{align}
\end{enumerate}
Assuming conditions  \ref{ass:reg}  and \ref{ass:logbdd}, 
 Theorem \ref{thm:MET} applies to the (continuous) linear cocycle $\Ac_{\uo, x}^n = D_x f^n_\uo$ over the RDS $f^n_\uo$ equipped with its stationary measure $\pi$. 

We now turn to the task of ruling out the degeneracy $d \lambda_1 = \lambda_\Sigma$. 
For the sake of simplicity, we state our results below in the special case $M = \T^d$, 
$d \geq 1$, where as usual the periodic box $\T^d$ is parametrized by $[0,2 \pi)^d$. 
Note that in this case, the tangent bundle $T M = T \T^d$ is parallelizable, i.e. it is diffeomorphic to the Cartesian product $\T^d \times \R^d$. 

Below, for $n \geq 1$ and $x \in \T^d$, the mapping $\Phi_x : \Omega_0^n \to \T^d$
is as defined previously. Additionally, we define $\widehat \Phi_x : \Omega_0^n \to SL_d(\R)$ by 
\begin{align}
\widehat \Phi_x(\uo^{n}) := \frac{1}{|\det D_x f^n_{\uo}|^{1/d}} D_x f^n_{\uo} \, , 
\end{align}
where here we follow the slightly nonstandard convention that $SL_d(\R)$ denotes the space of $d \times d$ matrices with determinant $\pm 1$. Observe that $\widehat \Phi_x:\Omega_0^n \to SL_d(\R)$ is a $C^1$ mapping. 


For $\uo^{n} \in \Omega_0^n$ and $x \in \T^d$, define 
\begin{align}
\Sigma_{x; \uo^{n}} := \ker D_{\uo^{n}} \Phi_{x} \subset T_{\uo^{n}} \Omega_0^n \, .
\end{align}

\begin{prop}\label{prop:suffCondRuleOutFurst}
Assume that $f_\omega$ satisfies conditions \ref{ass:reg}  and \ref{ass:logbdd} , and that the 
corresponding Markov kernel $P$ is uniformly geometrically ergodic
 with stationary measure $\pi$. 
Moreover, assume there exist $n \geq 1$ and 
$(\uo_\star^{n}, x_\star) \in \Supp(\rho_0) \times \Supp(\pi)$ which satisfy the following properties: 
\begin{enumerate} [label=(\roman*), ref=(\roman*)]
\item There are $c ,\eps > 0$ such that $\rho_0(\uo^{n}) \geq c > 0$ if $|\uo^{n}_\star - \uo^{n}| < \eps$. \label{item:Furstcond1}
\item The mapping $\Phi_{x_\star}$ is a submersion at $\uo^{n} = \uo^{n}_\star$. \label{item:Furstcond2}
\item The restriction of $D_{\uo^{n}_\star} \widehat \Phi_{x_\star}$ to $\Sigma_{x_\star; \uo^{n}}$ is surjective as a linear operator  onto $T_{\widehat \Phi_{x_\star} (\uo^{n}_\star)} SL_d(\R)$. \label{item:Furstcond3}
\end{enumerate}
Then, $d \lambda_1 > \lambda_\Sigma$. 
\end{prop}

\begin{rmk}\label{rmk:furstVolPresCase}
In our primary applications in this paper, the RDS $f^n_\uo$ is volume-preserving, 
and so $\det (D_x f^n_\uo) \equiv 1$. As a result, $\lambda_\Sigma = 0$ and the conditions of Proposition \ref{prop:suffCondRuleOutFurst} imply $\lambda_1 > 0$. 
\end{rmk}

The proof of Proposition \ref{prop:suffCondRuleOutFurst} uses the following two lemmas. The first is a consequence of the 
constant rank theorem. 
\begin{lem}\label{lem:constRankTheorem}
	Let $U \subset \R^a, V \subset \R^b, W \subset \R^c$ be open neighborhoods of the origin, with $a, b, c \in \N$. Let $F : U \to V, G : U \to W$ be $C^2$ mappings and let $\rho : U \to (0,\infty)$ be a strictly positive continuous function. Let $\dd u, \dd v, \dd w$ denote Lebesgue measure on $U, V, W$, respectively. Lastly, assume 
\begin{enumerate} [label=(\alph*), ref=(\alph*)]
		\item $D_{0} F$ is a surjection;  
		\item $D_{0} G|_{\ker D_{0} F}$ is a surjection. 
	\end{enumerate}
Then, there is a nonnegative continuous function $\widehat \rho : V \times W \to [0,\infty)$ such that $\widehat \rho \, \dd v \dd w = (F, G)_* [\rho \, \dd u]$
and $\widehat \rho > 0$ on a neighborhood of the origin.
\end{lem}

We also require the following. 

\begin{lem}\label{lem:emptyInterior}
Let $\eta, \eta'$ be measures on $S^{d-1}$. Then, 
\begin{equation}
M_{\eta, \eta'} = \{ A \in SL_d(\R) : A_* \eta = \eta'\}
\end{equation}
has empty interior in $SL_d(\R)$. 
\end{lem}
Below we provide a sketch with references for the proof of this standard fact. 
\begin{proof}[Sketched proof of Lemma \ref{lem:emptyInterior}]
In the case $\eta = \eta'$, we have that $M_\eta = M_{\eta, \eta}$ is a subgroup of $SL_d(\R)$, and by, e.g., \cite[Lemma 4.6]{bedrossian2018lagrangian}, we have the following cases: 
\begin{itemize}
	\item[(a)] $M_\eta$ is compact, in which case there is some fixed inner product $\langle \cdot, \cdot \rangle'$ on $\R^d$ with respect to which each element $A \in M_\eta$ is an isometry (c.f.  \cite[Lemma 6.7(ii)]{baxendale1989lyapunov}); and
	\item[(b)] $M_\eta$ is noncompact, in which case there is a collection of proper subspaces $V_1, \cdots, V_k$ of $\R^d$ such that for all $A \in M_\eta$, we have $A (\cup_i V_i) = \cup_i V_i$ (c.f. \cite[Theorem 8.6]{furstenberg1963noncommuting}). 
\end{itemize}
In either case, it is straightforward to show by a perturbation argument  that $M_\eta$ has empty interior. 
 
When $\eta \neq \eta'$, consider the mapping $\varphi : SL_d(\R) \times SL_d(\R) \to SL_d(\R)$ given by $\varphi(A, B) = A^{-1} B$. It is not hard to show that $\varphi$ is smooth, and that its derivative has constant rank $d^2-1$. In particular, $\varphi$ is an open mapping, i.e. the image of an open set is open. Moreover, we have that $\varphi(M_{\eta, \eta'} \times M_{\eta, \eta'}) \subset M_\eta$. This all implies that if $M_{\eta, \eta'}$ contained an open set, then $M_\eta$ would also contain an open set, a contradiction. Therefore $M_{\eta, \eta'}$ has empty interior. 
\end{proof}

\begin{proof}[Proof of Proposition \ref{prop:suffCondRuleOutFurst}]
Below we present the proof when $n = 1$, writing $\omega_\star = \uo_\star^{1}$. The general proof
is essentially the same. 
By Corollary \ref{cor:bougerol} and Lemma \ref{lem:emptyInterior}, it suffices to show that 
\begin{equation}
\big( \{ f_{\omega_\star} (x_\star)\} \times SL_d(\R) \big) \cap \Supp \left( \operatorname{Law}(\Phi_{x_\star} (\cdot), \widehat \Phi_{x_\star}(\cdot) ) \right) 
\end{equation}
contains a set of the form $\{ f_{\omega_\star}(x_\star)\} \times U$ where $U \subset SL_d(\R)$ is an open neighborhood. This follows from the hypotheses of Proposition \ref{prop:suffCondRuleOutFurst} and, on passing to local smooth charts, Lemma \ref{lem:constRankTheorem} applied to $F = \Phi_{x_\star}, G = \widehat \Phi_{x_\star}$. 
\end{proof}

\begin{rmk}
A version of Proposition \ref{prop:suffCondRuleOutFurst} can also be written down in the case when $M$ is a general compact Riemannian manifold. To do this one should work with the fiber bundle $SL(M)$ over $M$ with fibers 
\begin{equation}
SL_p(M) = \{ A : \R^n \to T_p M \text{ linear such that} \det (A) = \pm1\} \, , 
\end{equation}
regarded as a principle bundle over $M$ with structure group $SL_d(\R)$, see \cite{kobayashi1963foundations}. Note that $\det$ is defined using the standard inner product on $\R^d$ in the domain and the 
Riemannian inner product on $T_p M$ in the codomain. 
Fixing $x \in M$ and a determinant 1 isomorphism $\R^d \to T_x M$, we can now view the fiber $SL_p(M)$ over $p \in M$ as the space of determinant 1 mappings $A : T_x M \to T_p M$. With this convention, one can view $\widehat \Phi_x$ as a mapping 
from $\Omega_0^n$ to $SL(M)$, with
\begin{equation}
\widehat \Phi_x(\uo^{n}) = \left(\Phi_x(\uo^{n}), \frac{1}{|\det D_x f^n_{\uo}|^{1/d}} D_x f^n_{\uo} \right) \,. 
\end{equation}
With these identifications, the proof of Proposition \ref{prop:suffCondRuleOutFurst}
can now be carried out in charts and follows \emph{mutatis mutandi}. 
\end{rmk}


\section{Almost-sure mixing for incompressible RDS with absolutely continuous noise}\label{sec:mixing}

Equipped with tools for estimating Lyapunov exponents, we turn now to conditions 
ensuring almost-sure correlation decay. In the following 
informal discussion, we will assume that $f_\omega$ is a differentiable RDS on a compact manifold $M$ in the setting of Section \ref{sec:LE}, i.e. assuming conditions \ref{ass:reg} and \ref{ass:logbdd}). 
Additionally, we will assume the following: 

\begin{enumerate} [label=(U), ref=(U)]
	\item\label{ass:invar} $f_\omega$ admits stationary measure $\pi$ which is almost-surely invariant, i.e., $(f_\omega)_* \pi = \pi$
	for $\P_0$-almost every $\omega \in \Omega_0$. 
\end{enumerate} 

Assuming \ref{ass:invar}, almost-sure mixing asserts decay of the correlation terms 
\begin{equation}
\Cor_n(\varphi, \psi) := \left| \int \varphi(x) \psi \circ f^n_\uo(x) \dd \pi(x) \right| \quad \text{ with probability 1}
\end{equation}
as $n \to \infty$ for sufficiently regular observables $\varphi,\psi$ with $\pi$-mean zero. 
The term $\Cor_n(\varphi, \psi)$ depends on the random sample $\uo$; we can estimate the probability that 
this quantity is not too large by using Chebyshev's inequality, i.e. 
\begin{equation}
\P \left\{ \Cor_n(\varphi, \psi) > \varepsilon \right\} 
\leq\varepsilon^{-2}\, \E_{\P } \left| \int  \varphi \psi \circ f^n_\uo \dd \pi \right|^2.
\end{equation}
The numerator on the right-hand side is equal to 
\begin{equation}\label{eq:twopt}
\E_{\P } \left| \int  \varphi \psi \circ f^n_\uo \,  \dd \pi \right|^2 = 
\E_{\P} \int  \varphi(x) \varphi(y) \psi \circ  f^n_\uo(x) \psi \circ f^n_\uo(y)  \dd  \pi(x) \dd  \pi(y) \, , 
\end{equation}
having introduced the dummy integration variable $y \in M$. We can rewrite this expression using the two-point process $(x_n, y_n)$ on $M \times M$ given for fixed initial $x_0, y_0$ by 
\begin{equation}
(x_n, y_n) = \left(f^n_\uo (x_0), f^n_\uo(y_0)\right) .
\end{equation}
In this notation, the expression in \eqref{eq:twopt} is given by
\begin{equation}
\E_{\P } \left| \int  \varphi \psi \circ f^n_\uo  \, \dd \pi \right|^2 = 
\int \psi^{(2)}  (P^{(2)})^n \varphi^{(2)} \dd \pi^{(2)} \, . 
\end{equation}
Here $\varphi^{(2)}(x,y) := \varphi(x) \varphi(y)$, with $\psi^{(2)}$ similarly defined,
while $\dd \pi^{(2)} (x,y) = \dd \pi (x) \dd \pi(y)$, and $P^{(2)}$ denotes the Markov semigroup 
for $(x_n, y_n)$. 
We conclude from this calculation the following general principle: 
\begin{gather*}
\textit{Geometric ergodicity of the two-point Markov semigroup $P^{(2)}$} \\
\textit{implies exponential mixing with probability 1 for $f_\omega$.}
\end{gather*}
This principle and the argument presented above 
was known to the Sinai school in the early 2000's, see, e.g., the 
paper \cite{dolgopyat2004sample} or the more recent work \cite{bedrossian2019almost}. 

On the other hand, we 
emphasize that $P^{(2)}$ cannot possibly be
uniformly geometrically ergodic with respect to $\pi^{(2)}$, since the diagonal
\begin{equation}
\Delta = \{ (x, x) : x \in M\} \subset M \times M
\end{equation}
 is almost-surely invariant under the action of $(x, y) \mapsto (f_\omega(x), f_\omega(y))$, and so supports at least one additional stationary measure for $P^{(2)}$ not equal to $\pi^{(2)}$. 
 Instead, we must work with the noncompact
phase space 
\begin{equation}
M^{(2)} := M \times M \setminus \Delta \, , 
\end{equation}
 which necessitates the use
of a Lyapunov function $V : M^{(2)} \to [1,\infty)$ as in Harris' theorem (Theorem \ref{thm:Harris})
to ensure positive recurrence of $P^{(2)}$ away from $\Delta$. 

The construction of such a $V$ is the primary goal of this section. 
In Section \ref{subsec:projTwisted}, we collect some preliminary results which 
connect Lyapunov exponents for $f_\omega$ to growth of moments of $|D_x f^n_\uo|$. This is useful to prove repulsion 
from $\Delta$: if $x,y$ are sufficiently close, then 
\begin{equation}
f^n_\uo(y) \approx f^n_\uo(x) + D_x f^n_\uo(y - x) \,,
\end{equation}
and so growth of $|D_x f^n_\uo|$ implies some repulsion from $\Delta$, at least at the level of the linearization. 
This is promoted to \emph{nonlinear} repulsion in 
Section \ref{subsec:driftCondFromLE}, where we define the drift function $V$
used in our proof. Finally, in Section \ref{subsec:quenchedMixing}, we apply Harris' theorem to $P^{(2)}$ and fill in the details needed to establish quenched mixing for $f_\omega$ in Proposition \ref{prop:twoPointToMixing}. 



\subsection{Projective and twisted semigroups}\label{subsec:projTwisted}

In this subsection, we define several auxiliary Markov semigroups and affiliated linear operators for use in our construction of the desired Lyapunov function $V$. 

Conceptually, the form of the function $V$ depends most on its values near 
the diagonal $\Delta$. Here, we approximate the two-point motion by its linearization. 
This motivates the use of the \emph{linear semigroup} $TP$ acting on $TM$, defined
for bounded, measurable $\varphi : TM \to \R$ by
\begin{equation}
T P \varphi(x, u) := \E_{\P_0} \varphi(f_\omega (x), D_x f_\omega u) \qquad x \in M, u \in T_x M \,. 
\end{equation}
Of particular use are functions on $TM$ of the form
\begin{equation}
\varphi(x, u) = |u|^{-q} \psi(x, u/|u|) \, , \quad q > 0 \, , 
\end{equation}
where $\psi : S M \to \R$ is measurable and bounded. Here and throughout, 
$SM$ denotes the unit tangent bundle of $M$ consisting of elements $(x, v) \in TM$ such that $|v| = 1$. Restricting $TP$ 
to such observables, we see that
\begin{equation}
T P \varphi (x, u) = \widehat P_q \psi (x, u/|u|) \, , 
\end{equation}
where $\widehat P_q$ is the \emph{twisted projective semigroup} defined for bounded, measurable observables $\psi : S M \to \R$ by
\begin{equation}
\widehat P_q \psi(x, v) = \E_{\P_0} |D_x f_\omega v|^{-q} \psi(f_\omega(x), D_x f_\omega v/|D_x f_\omega v|) \,.
\end{equation}
The closely-related $\widehat P = \widehat P_0$ is the \emph{projective semigroup}, a Markov semigroup on $SM$ corresponding to the Markov chain $(x_n, v_n)$ defined
for fixed initial $(x_0, v_0) \in SM$ by
\begin{equation}
(x_n, v_n) = (f^n_\uo(x_0), D_{x_0} f^n_\uo(v_0) / |D_{x_0} f^n_\uo(v_0)| ) \,. 
\end{equation}
%
%
%
We now turn to the spectral theory of $\widehat P, \widehat P_q$, which we will later use to 
construct Lyapunov functions for $P^{(2)}$. For a $d \times d$ invertible matrix $A$ we write $m(A) = | A^{-1}|^{-1}$. 
The following 
assumption is useful and simplifies arguments. 

\begin{enumerate} [label=(B), ref=(B)]
	\item\label{ass:jacobbd} There is a constant $C_0 > 0$ such that
	\begin{equation}
	m(D_x f_\omega) \geq C_0^{-1} > 0 \quad \text{ and } \quad | D_x f_\omega| \leq C_0 
	\end{equation}
	with probability 1. Moreover, $\| f_\omega\|_{C^2}$ is $\P_0$-essentially bounded. 
\end{enumerate} 

Note that condition \ref{ass:jacobbd} implies condition \ref{ass:logbdd} from Section \ref{sec:LE}. 
Below, the semigroups $\widehat P, \widehat P_q$ are regarded as bounded linear operators on 
$C^0(SM)$, with $\| \cdot \|$ denoting the uniform norm. Their spectrum will be denoted by $\sigma(\widehat P),\sigma(\widehat P_p)$, respectively.
\begin{lem}\label{lem:specPic}
Assume conditions \ref{ass:reg} and \ref{ass:jacobbd} hold true. Moreover, 
assume $\widehat P$ is uniformly geometrically ergodic with respect to a 
(unique) stationary measure $\widehat \pi$ on $SM$. Then, there exists $q_0 > 0$ such that 
for all $q \in [-q_0, q_0]$, it holds that 
 $\widehat P_q$ admits a simple dominant eigenvalue $r(q) > 0$ such that 
 $\sigma(\widehat P_q) \setminus \{ r(q)\}$ is contained in a ball of radius less than $r(q)$. 
Consequently, for all such $q$ there is a unique dominant eigenfunction
  $\psi_q$ such that 
\begin{equation}
\lim_{n \to \infty} \| \psi_q - r(q)^{-n} \widehat P_q^n {\bf 1} \| = 0 \,. 
\end{equation}
\end{lem}
Here, we have written ${\bf 1}$ for the constant function identically equal to 1. 
The value $\Lambda(q) = \log r(-q)$ is often referred to as the \emph{moment Lyapunov exponent}, and is related to large deviations estimates in the convergence of Lyapunov exponents \cite{arnold1984formula}. 

\begin{proof}
As $\widehat P$ is uniformly geometrically ergodic, the eigenvalue $1$ is simple and 
has the property that $\sigma(\widehat P) \setminus \{ 1 \}$ is contained in a ball of radius $\delta < 1$. To complete the proof, it suffices by classical spectral theory~\cite{kato2013perturbation}, 
to check that $\widehat P_q \to \widehat P$ in norm as $q \to 0$. For $\psi : SM \to \R$, we have
\begin{equation}
|\widehat P_q \psi(x, v) - \widehat P \psi (x, v)| \leq \| \psi\| \cdot \E_{\P_0} \big| | D_x f_{\omega} v|^q - 1 \big|  \,. 
\end{equation}
To estimate this, we observe that for $a > 0$ and $q \in \R$, we have
\begin{equation}
|a^q - 1| \leq |q(a-1)| \cdot \max\{ 1, a^{q-1}\} \,. 
\end{equation}
By \ref{ass:jacobbd}, it follows that 
\begin{equation}
\big| | D_x f_{\omega} v|^q - 1 \big| \leq |q| (C_0 + 1) C_0^{1+|q|},
\end{equation}
on handling each of the cases $q > 1, q < 1$ separately. The conclusion of the lemma now follows. 
\end{proof}

The following is used to control the value of $r(q)$. 
\begin{lem}\label{lem:estMomLE}
Under the hypotheses of Lemma \ref{lem:specPic}, we have that $q \mapsto r(q)$ is  continuously differentiable over $q \in (-q_0, q_0)$ and satisfies
\begin{equation}
\frac{\dd}{\dd q} r(0) = - \lambda_1 \, ,
\end{equation}
where $\lambda_1$ is the Lyapunov exponent of the derivative cocycle $D_x f^n_\uo$. Consequently, if $\lambda_1 > 0$ then $r(q) < 1$ for all sufficiently small $q > 0$. 
\end{lem}
The proof is straightforward; see, for example, \cite{arnold1984formula} or \cite[Lemma 5.10]{bedrossian2019almost} for the proof in a similar setting. 
%
%
In the construction to come, we will use the following additional properties of 
$\psi_p$. 

\begin{cor}\label{cor:propsPsiQ}
Assume the hypotheses of Lemma \ref{lem:specPic}. 
\begin{enumerate} [label=(\roman*), ref=(\roman*)]
\item\label{cor:propsPsiQ:item1} For any $\eps > 0$ we have that $\psi_q = \psi_q^{C^0} + \psi_q^{C^1}$, where
$\| \psi_q^{C^0} \| \leq \eps$ while $\psi_q^{C^1}$ is continuously differentiable and satisfies $\| \psi_q^{C^1} \| + \| D \psi_q^{C^1} \| \lesssim_\eps 1$.

\item\label{cor:propsPsiQ:item2} If, in addition, the Markov chain $(x_n, v_n)$ is topologically irreducible, then $\psi_q$ is strictly positive on SM. 
\end{enumerate}
\end{cor}
\begin{proof}
	Item \ref{cor:propsPsiQ:item1} is an immediate consequence of the $C^0$-convergence $r(q)^{-n} \widehat P^n_q {\bf 1} \to \psi_q$ and the fact that $\widehat P^n_q {\bf 1}$ is a $C^1$ function for all $n \geq 1$
	(this uses assumption \ref{ass:reg}). 
	
	For item \ref{cor:propsPsiQ:item2}, there is an open neighborhood $U \subset SM$ such that $\psi_q|_U \geq c > 0$ for some constant $c > 0$.  
	Now, given $(x, v) \in SM$, by assumption \ref{ass:jacobbd} we have
	\begin{equation}
	\psi_q(x, v) \geq C_0^{- n|q|} r(q)^{-n} \E_{(x, v)} \psi_q(x_n, v_n) 
	\geq c C_0^{-n |q|} r(q)^{-n} \widehat P^n((x,v), U) 
	\end{equation}
	for all $n \geq 1$. By topological irreducibility, there is some $n$ such that $P^n((x,v), U) > 0$, completing the proof. 
\end{proof}

\subsection{Drift condition for the two-point process}\label{subsec:driftCondFromLE}

We now set about constructing the drift function $V$ for the two-point process. Throughout this discussion, we take on all the hypotheses of Lemma \ref{lem:specPic} as well as topological irreducibility 
for $(x_n, v_n)$ as in Corollary \ref{cor:propsPsiQ}\ref{cor:propsPsiQ:item2}. 

\medskip

\noindent {\bf Notation.} 
Firstly, for $s > 0$, we define
\begin{equation}
\Delta(s) := \{ (x, y) \in M^{(2)} : d(x,y) < s \} \, .
\end{equation}
Let $s_0 > 0$ denote the minimal injectivity radius of $\exp_x, x \in M$ and 
$w : \Delta(r) \to TM$, $w(x, y) := \exp_x^{-1}(y)$, noting that if $d(x,y) < s_0$ then 
$\exp_x^{-1} (y) \in T_x M$ is uniquely defined.  Set $\widehat w(x,y)$ to be $w(x,y) / |w(x,y)|$. 

\medskip

Fix a parameter $p \in (0,q_0)$ so that $r(p) < 1$ (which is possible due to Lemma~\ref{lem:estMomLE}), and fix a value $s < s_0$, to be determined as we proceed. Our drift function will be of the form
\begin{equation}
V(x,y) = \chi(x,y) \, d(x,y)^{-p} \psi_{p}(x, \widehat w(x,y)) + \hat c (1 - \chi(x,y)) \, , 
\end{equation}
where $\chi$ is a bump function supported near the diagonal, so that 
$\chi(x,y) \equiv 1$ on $\Delta(s/2)$ and $\chi(x,y) \equiv 0$ on $\Delta(3 s/4)^c$. Here, $\hat c > 0$ is a constant chosen so that $V(x,y) \geq 1$ for all $(x,y) \in M^{(2)}$.  

We now set about proving the desired drift condition. For this we will use the following
approximation lemma applied to the function $\widehat V : T M \to \R$ defined by
\begin{equation}\eps
\widehat V (x,v) := |v|^{-p} \psi_{p}(x, v) \,. 
\end{equation}
\begin{lem} \label{lem:linearizationEstimate}
 Assume the setting at the beginning of Section \ref{subsec:driftCondFromLE}. Assume $s$ is sufficiently small, depending only on the $\P_0$-essential supremum of $\| f_\omega\|_{C^2}$, as in condition \ref{ass:jacobbd}. Then, for any $\eps > 0$, there is a constant $C = C_\eps$ such that 
\begin{equation}
	| T P \widehat V(x, w(x,y)) - P^{(2)} V(x,y) | \lesssim  \eps d(x,y)^{-p} + C_\eps d(x,y)^{1 - p},
\end{equation}
	for all $(x,y) \in \Delta(s / (2 C_0))$, uniformly in $\eps$. Here, $C_0$ is
	the constant appearing in condition \ref{ass:jacobbd}. 
\end{lem}
\begin{proof}
Given $\eps > 0$, let $\psi_{p} = \psi_{p}^{C^0} + \psi_{p}^{C^1}$ be as in Corollary \ref{cor:propsPsiQ}\ref{cor:propsPsiQ:item1}, noting that $\|\psi_{p}^{C^0}\|_{C^0} \leq \eps$ and $\| \psi_{p}^{C^1}\|_{C^1} \leq C_\eps < \infty$. 
We now collect the estimates we use. 
From the first order Taylor expansion on manifolds and condition  \ref{ass:jacobbd}, we have that 
\begin{equation}
d_M ( f_\omega (y), \exp_{f_\omega (x)} D_x f_\omega(w(x,y)) ) \lesssim |w(x,y)|^2 ,
\end{equation}
for all $x, y \in \Delta(s)$, with $s$ taken sufficiently small in terms of $\operatorname{EssSup} \| f_\omega\|_{C^2}$. It follows that
\begin{align}
| w(f_\omega (x), f_\omega (y) ) - D_x f_\omega(w(x, y))|  \lesssim |w(x,y)|^2 \,, 
\end{align}
which in conjunction with the mean value theorem for $a \mapsto a^{-p}$ implies  
\begin{align}\label{eq:estPthPower}
\left| | w(f_\omega (x), f_\omega (y))|^{-p} - |D_x f_\omega(w)|^{-p} \right| 
\lesssim p |w|^{-p-1} |w|^2 \leq |w|^{1 - p} \, ,
\end{align}
for $w = w(x,y)$.
Lastly, using \ref{ass:jacobbd} once more, we estimate
\begin{align}\label{eq:estProj}
\left| \widehat w(f_\omega x, f_\omega y) - \frac{D_x f_\omega(w)}{|D_x f_\omega(w)|} \right| 
\lesssim \frac{|w|^2}{|w|} = |w| ,
\end{align}
using the following general inequality for vectors $v, v'$: 
\begin{align}
\left| \frac{v}{|v|} - \frac{v'}{|v'|} \right| \leq 2 \frac{|v - v'|}{\min\{ | v|, |v'|\}} .
\end{align}
For the main estimate, we start by noting that 
if $(x,y) \in \Delta(s / (2 C_0))$, then $(f_\omega (x), f_\omega (y)) \in \Delta(s/2)$, and so
\begin{align}
P^{(2)} V(x, y) & = \E |w(f_\omega (x), f_\omega (y))|^{-p} \psi_{p}(x, \widehat w(f_\omega (x), f_\omega (y)))\notag\\
 & = \E |w(f_\omega (x), f_\omega (y))|^{-p} \psi_{p}^{C^1} (x, \widehat w(f_\omega (x), f_\omega (y))) + O( \eps d(x,y)^{-p}).
\end{align}
Meanwhile, we have that 
\begin{align}
T P \widehat V(x,w) & = \E |D_x f_\omega(w)|^{-p} \psi_{p} (x, D_x f_\omega(w) / |D_x f_\omega(w)|) \, , \notag\\
& = \E |D_x f_\omega(w)|^{-p} \psi_{p}^{C^1} (x, D_x f_\omega(w) / |D_x f_\omega(w)|) + O(\eps d(x,y)^{-p}) \, .
\end{align}
Combining the above estimates, we obtain
\begin{align}
P^{(2)} V(x,y) - TP \widehat V(x, w) & = O\left( \eps d(x,y)^{-p} + \| \psi_p^{C^1}\|_{C^0} d(x,y)^{1-p} 
+  \| \psi_p^{C^1}\|_{C^1}  d(x,y)^{1-p} \right)\notag \\
& = O \left( \eps d(x,y)^{-p} + C_\eps d(x,y)^{1-p} \right),
\end{align}
as desired. 
\end{proof}

We collect this into the following drift-type estimate. 

\begin{prop}\label{prop:twoPointDrift}
Assume the setting at the beginning of Section \ref{subsec:driftCondFromLE}, and moreover, that the Lyapunov exponent $\lambda_1$ of $(f^n_\uo)$ is positive. Let $p$ be sufficiently small so that $r(p) < 1$ (Lemma \ref{lem:estMomLE}). 
Then, there exist $\gamma < 1$ and $s=s_* > 0$ so that 
\begin{equation}
P^{(2)} V(x,y) < \gamma V(x,y) \qquad \forall (x,y) \in \Delta(s_*) \, .
\end{equation}
\end{prop}
Note that since $\Delta(s_*)^c \subset M^{(2)}$ is compact, the above estimate implies
that $V$ satisfies a drift condition for $P^{(2)}$ as in Definition \ref{defn:driftCond}.
\begin{proof}
Let $\eps > 0$ be sufficiently small, constraints on which will be made as we go along. 
Apply Lemma \ref{lem:linearizationEstimate}, and observe that 
for $(x,y) \in \Delta(s / (2 C_0))$ we have that
\begin{align}
P^{(2)} V(x,y) \leq TP \widehat V(x, w(x,y)) + K ( \eps d(x,y)^{-p} + C_\eps d(x,y)^{1-p}) \, .
\end{align}
where $K > 0$ is a constant not depending on $\eps$. By our construction, 
$\widehat V(x, w(x,y))$ is an eigenfunction for $TP$ and satisfies the relation 
\begin{equation}
TP \widehat V(x,w(x,y)) = r(p) \widehat V(x,w(x,y)) = r(p) V(x,y) = r(p) d(x,y)^{-p} \psi_p(x, \widehat w(x,y))
\end{equation}
 Plugging this in and assuming $(x,y) \in \Delta(s_*)$, for some $s_* > 0$ to be determined, 
\begin{equation}
 P^{(2)} V(x,y) \leq d(x,y)^{-p} \left( r(p) \psi_p(x,\widehat w(x,y)) + K \eps + K C_\eps s_* \right).
\end{equation}
We now specify the constants $\eps, s_*, \gamma$ required to complete the proof. To start, set 
\begin{equation}
\eps = \frac{1}{100 K} (1 - r(p)) \inf_{(x,v) \in SM} \psi_p(x,v) \, , 
\end{equation}
 which is $> 0$ by Corollary \ref{cor:propsPsiQ}\ref{cor:propsPsiQ:item2}.  With this value of $\eps$ fixed, 
define $s_* = \eps/C_\eps$. Plugging these choices into the above bound 
for $P^{(2)}V$, we see that
\begin{equation}
 P^{(2)} V(x,y) \leq  d(x,y)^{-p} \left(r(p) + \frac{1}{50} (1 - r(p)) \right) \psi_p(x,\widehat w(x,y)) = \gamma V(x,y) ,
\end{equation}
for $(x,y) \in \Delta(s_*)$, where $\gamma$, defined to be the parenthetical term, is automatically less than 1. 
\end{proof}


\subsection{Deducing scalar mixing}\label{subsec:quenchedMixing}

The following connects geometric ergodicity of the two-point process to 
almost-sure correlation decay for the one-point process. 

\begin{prop}\label{prop:twoPointToMixing}
Assume $f_\omega$ is a continuous RDS on a compact, orientable, Riemannian
manifold $M$ without boundary satisfying condition \ref{ass:invar}. 
Assume that the two-point process with kernel $P^{(2)}$ on $M^{(2)}$ is $V$-geometrically ergodic, 
where $V : M^{(2)} \to [0,\infty)$ is integrable with respect to $\pi^{(2)} = \pi \times \pi$. 
Let $q, s > 0$ be fixed. 

Then, there is a function $D = D_{q,s} : \Omega \to [1,\infty)$ and  a constant $\alpha = \alpha_{q,s} > 0$ such that for all 
$H^s$ functions $\varphi, \psi : M \to \R$ of $\pi$-mean zero, we have that
\begin{equation}\label{eq:almostSureMixing4}
\left| \int \varphi(x) \psi \circ f^n_\uo(x) \dd \pi(x) \right| \leq D(\uo) \ee^{- \alpha n} \| \varphi\|_{H^s} \| \psi\|_{H^s} \, , 
\end{equation}
while the function $D$ satisfies $\E_\P D^q < \infty$. 
\end{prop}

The proof in the case $M = \T^d$ is essentially contained in \cite[Section 7]{bedrossian2019almost}, which we sketch here briefly for the sake of completeness.
 See Remark \ref{rmk:genMfl4}
for comments on the case when $M$ is a general compact manifold without boundary.  

\begin{proof}[Proof sketch of Proposition \ref{prop:twoPointToMixing} when $M = \T^d$]
We present a sketch here assuming that the density $\frac{\dd \pi}{\dd \Leb} \equiv 1$; the general case in condition \ref{ass:invar} is handled in essentially the same way. 

Let $\{ e_k, k \in \Z^d\}$ denote the orthogonal basis $e_k(x) = \ee^{i k \cdot x}$ for $L^2(\T^d)$ and let $\varphi, \psi : \T^d \to \R$ be smooth, mean-zero functions with Fourier expansions $\varphi = \sum_{k \in \Z^d_0} \varphi_k e_k, \psi = \sum_{k \in \Z^d_0} \psi_k e_k$. Here, $\Z^d_0 := \Z^d \setminus \{ 0 \}$, noting that $\varphi, \psi$ are assumed to be mean-zero. Lastly, we will assume that 
\begin{equation}
\| (P^{(2)})^n - 1 \otimes \pi^{(2)} \|_V \leq C \ee^{- \beta n}\, , \quad \beta > 0 \, , 
\end{equation}
with notation as in Harris' theorem (Theorem \ref{thm:Harris}). 

To start, fix $\zeta > 0$, to be specified as we go along, and for $k, k' \in \Z^d_0$ define
\begin{equation}
N_{k,k'} = \max\left\{ n \geq 0 : \left| \int e_k(x) e_{k'} \circ f^n_\uo(x) \, \dd x \right| > \ee^{- \zeta n} \right\} \, , 
\end{equation}
noting by the Chebyshev inequality argument presented at the beginning of Section \ref{sec:mixing} that 
\begin{equation}
\P( N_{k, k'} > \ell) \lesssim \ee^{\ell ( 2 \zeta - \beta)} \, .
\end{equation}
Note that in this step we used that $V \in L^1(\pi^{(2)})$. 
Hereafter we assume $\zeta <  \beta/2$, so in particular $N_{k, k'}$ is almost-surely finite, and observe that we have the estimate
\begin{equation}\label{eq:firstest}
\left| \int e_k(x) e_{k'} \circ f^n_\uo(x) \, \dd x \right|  \leq \ee^{\zeta N_{k, k'}} \ee^{- \zeta n} \,, \quad \text{hence} \quad \left| \int \varphi(x) \psi \circ f^n_\uo(x) \dd x \right| \leq \ee^{- \zeta n} \sum_{k, k'} | \varphi_k| |\psi_{k'}| \ee^{\zeta N_{k, k'}} \,. 
\end{equation}
Define the random variable
\begin{equation}
K = \max \{ |k| \vee |k'| : \e^{\zeta N_{k, k'}} > |k| |k'|\}
\end{equation}
where $a \vee b = \max\{ a, b\}$ for $a, b \in \R$, and observe that
\begin{align}
\P(K > \ell) & \leq 2 \sum_{\substack{k,k' \in \Z^d_0 \\ |k| > \ell}}
\P \left( \ee^{\zeta N_{k, k'}} > |k| |k'|\right) 
 \lesssim \sum_{\substack{k \in \Z^d_0 \\ |k| > \ell}} |k|^{\frac{2 \zeta - \beta}{\zeta}}
 \lesssim \ell^{d + \frac{2 \zeta - \beta}{\zeta}} \, , 
\end{align}
assuming, as we shall going forward, that $\zeta$ is small enough so that $d + \frac{2 \zeta - \beta}{\zeta} < 0$. The random variable $K$ is almost-surely finite, as is the random variable
\begin{equation}
\widehat D = \max_{|k|, |k'| \leq K} \ee^{\zeta N_{k, k'}} \, . 
\end{equation}
Noting that $\ee^{\zeta N_{k, k'}} \leq \widehat D  |k| |k'|$ unconditionally, we conclude from \eqref{eq:firstest} that 
\begin{align}
\left| \int \varphi(x) \psi \circ f^n_\uo(x) \dd x \right| 
&\leq \widehat D \ee^{- \zeta n} \left(\sum_{k} |k| |\varphi_k| \right) \left(\sum_{k'} |k'| |\psi_{k'}| \right)  \leq \widehat D \ee^{- \zeta n} \| \varphi\|_{H^{\frac{d}{2} + 2}} \| \psi\|_{H^{\frac{d}{2} + 2}} \,. 
\end{align}
This is our desired estimate, except for the fact that on the right-hand side we bound in terms of $H^{\frac{d}{2} + 2}$. Going from here to $H^s$ regularity for arbitrary $s > 0$ follows from an approximation argument and requires shrinking the mixing rate $\zeta$; see, for example,  \cite[Lemma 7.1]{bedrossian2019almost} or  \cite[Lemma 4.2]{zelati2020relation}. 
\end{proof}

\begin{rmk}\label{rmk:genMfl4}
To handle the case when $M$ is a general compact manifold without boundary, let $U_1, \dots$ $,U_m$ be a finite cover of $M$ by smooth charts $(\phi_i, U_i)$ with $\phi_i : U_i \to \R^d$, $d = \dim M$. Assume, as we may, that the $\phi_i$ have bounded range in $\R^d$, and let $(\chi_i)$ be a smooth partition of unity with $\chi_i$ supported on each $U_i$. Now,  \cite[Theorem 7.5.1]{triebel1994theory} implies that the expression
\begin{equation}
\bigg( \sum_i \| \underbrace{\chi_i \varphi \circ \phi_i^{-1}}_{=: \varphi_i} \|_{H^s}^2 \bigg)^{1/2}
\end{equation}
is an equivalent norm for the $H^s$ norm of $\varphi$ on $M$. From here, one can embed each $U_i$ in the periodic box $\T^d$, obtaining sequences of functions $e_{i, k}, k \in \Z^d_0$ such that under the expansion $\varphi = \sum_{i, k} \varphi_{i, k} e_{i, k}$ with coefficients $\varphi_{i, k} \in \C$, we have
\begin{equation}
\| \varphi\|_{H^s}^2 \approx \sum_{i, k} |k|^{2s} |\varphi_{i, k}|^2 \,. 
\end{equation}
The proof of Proposition \ref{prop:twoPointToMixing} now translates to this setting, 
throughout working with the sequences $\{ e_{i, k}\}$ over indices $i, k$ instead of $\{ e_k\}$ over $k$. We omit any further details.
\end{rmk}

\subsection{Summary of sufficient conditions for almost-sure exponential mixing}\label{sec:checklist}

The paper thus far culminates in Proposition \ref{prop:twoPointToMixing}, which provides a sufficient condition for an RDS to be exponentially mixing with probability 1. The following is a summary of how, for a given RDS, one can check these sufficient conditions using the tools developed up until this point. In the following section, these conditions will be checked for the Pierrehumbert model introduced in Section \ref{sec:intro}.

  \begin{enumerate}[wide, labelwidth=!, labelindent=0pt, label=\textbf{(\arabic*})]
  \setcounter{enumi}{-1}
    \item\label{suffcond0}  {\bf Basic assumptions:} 
    Let $M$ be a compact Riemannian manifold without boundary and $(\Omega_0, \Fc_0, \P_0)$ a probability space, $(\Omega, \Fc, \P) := (\Omega_0, \Fc_0, \P_0)^{\Z_{\geq 1}}$. Let $\omega \mapsto f_\omega$ be a measurable assignment to each $\omega \in \Omega_0$ of a continuous mapping $f_\omega : M \to M$, with compositions 
\begin{equation}
f^n_\uo = f_{\omega_n} \circ \cdots \circ f_{\omega_1} 
\end{equation}
$\uo  = (\omega_1, \omega_2, \cdots) \in \Omega$. 
In addition, $f_\omega$ will be assumed to satisfy the following:
 
 \medskip

\begin{itemize}[leftmargin=42pt]
	\item[\ref{ass:reg}] The IID noise parameters $\omega_i, i \geq 1$ come from a probability space 
	$(\Omega_0, \P_0) = (\R^k, \rho_0 \dd \Leb)$ for some density $\rho_0 : \R^k \to \R_{\geq 0}$, and the mapping $(\omega, x) \mapsto f_\omega (x)$ is $C^2$ differentiable; 
	\item[\ref{ass:jacobbd}] There is a constant $C_0 > 0$ such that
\begin{equation}
	m(D_x f_\omega) \geq C_0^{-1} > 0 \quad \text{and} \quad | D_x f_\omega| \leq C_0 \, , 
\end{equation}
	almost-surely, and $\| f_\omega\|_{C^2}$ is $\P_0$-essentially bounded; 
\end{itemize}
and the following mild strengthening of \ref{ass:invar}: 

\medskip

\begin{enumerate}[label=(U'), ref=(U'),leftmargin=42pt]
	\item\label{ass:altUp} For $\P_0$-a.e. $\omega$, we have that $f_\omega$ preserves $\pi = \Leb$ on $M$. 
\end{enumerate}

\medskip

    \item\label{suffcond1}    {\bf  Conditions for the one-point chain:}
 The one-point chain is the Markov chain on $M$ with transition kernel 
$P(x, K) := \P_0( f_\omega (x) \in K)$. We will assume
\begin{enumerate} [label=(UE-$P$), ref=(UE-$P$), leftmargin=42pt]
	\item\label{ass:UEP} The kernel $P$ is uniformly geometrically ergodic with unique stationary measure $\pi$. 
\end{enumerate} 
By Theorem \ref{thm:Harris} and compactness of $M$, this follows on showing 
that $P$ is topologically irreducible, admits an open small set (Proposition \ref{prop:Tchain3}), and is strongly aperiodic (Lemma \ref{lem:aperiodicSuffCond}). 

\medskip
   
    \item\label{suffcond2}     {\bf Positive Lyapunov exponent:}
By (0) and (1), Theorem \ref{thm:MET} applied to the cocycle $\Ac^n_{\uo, x} = D_x f^n_\uo$ implies that the asymptotic exponential growth rate $\lambda_1 := \lim_{n \to \infty} \frac1n \log | D_x f^n_\uo|$ exists and is constant over $\P \times \pi$-a.e. $(\uo, x)$.
We will assume

\medskip

\begin{enumerate} [label=(LE), ref=(LE), leftmargin=42pt]
	\item\label{ass:LE} $\lambda_1 > 0$.
\end{enumerate} 

\medskip

\noindent Proposition \ref{prop:suffCondRuleOutFurst} gives a sufficient condition for $d \lambda_1 > \lambda_\Sigma$, where $d = \dim M$, while almost-sure volume preservation as in \ref{ass:altUp} implies $\lambda_\Sigma = 0$ (see Remark \ref{rmk:sumLE3}). 

\medskip

    \item\label{suffcond3}    {\bf Conditions for the projective chain:}
The projective chain is the Markov chain on the sphere bundle $S M$ with 
transition kernel $\widehat P ((x, v) , K) = \P_0 ( \widehat f_\omega(x, v) \in K)$, where
for $(x, v) \in S M$, 
\begin{equation}
\widehat f_\omega(x, v) := \left( f_\omega (x), \frac{D_x f_\omega v}{|D_x f_\omega v|} \right) \, . 
\end{equation}
We will assume

\medskip

\begin{enumerate} [label=(UE-$\widehat P$), ref=(UE-$\widehat P$), leftmargin=42pt]
	\item\label{ass:UEPhat} The kernel $\widehat P$ is uniformly geometrically ergodic with unique stationary measure $\widehat \pi$. 
\end{enumerate} 

\medskip

\noindent By Theorem \ref{thm:Harris} and compactness of $SM$, this follows on showing 
that $\widehat P$ is topologically irreducible, admits an open small set (Proposition \ref{prop:Tchain3}), and is strongly aperiodic (Lemma \ref{lem:aperiodicSuffCond}). 

\medskip

    \item\label{suffcond4}   {\bf Conditions for the two-point chain:}
The two-point chain is the Markov chain on $M^{(2)} := M \times M \setminus \Delta, \Delta := \{ (x,x) : x \in M\} \subset M \times M$ with transition kernel 
\begin{equation}
P^{(2)} ((x,y), K) = \P_0 ( (f_\omega(x), f_\omega(y)) \in K) \, .
\end{equation}
By condition \ref{ass:altUp}, the measure $\pi^{(2)} := \pi \times \pi$ is almost-surely invariant under $f_\omega \times f_\omega$, hence stationary for $P^{(2)}$. We will assume: 

\medskip 

\begin{enumerate} [label=(UE-$P^{(2)}$), ref=(UE-$P^{(2)}$), leftmargin=42pt]
	\item\label{ass:UEP2} The kernel $P^{(2)}$ is $V$-uniformly geometrically ergodic, where $V : M^{(2)} \to \R$ is of the form
\begin{equation}\label{eq:formOfVChecklist}
	V(x,y) = d(x,y)^{-p} \psi(x,y),
\end{equation}
	for some $p > 0$ small, where $\psi : M^{(2)} \to(0,\infty)$ is continuous, bounded from above, and bounded from below by a constant $c > 0$ on a small neighborhood of $\Delta$. 
\end{enumerate} 

\medskip

As before, we must check in Theorem \ref{thm:Harris} that $P^{(2)}$ is topologically irreducible, admits an open small set (Proposition \ref{prop:Tchain3}), and is strongly aperiodic (Lemma \ref{lem:aperiodicSuffCond}). Additionally, since $M^{(2)}$ is noncompact it is necessary 
to build $V$ as above satisfying a drift condition (Definition \ref{defn:driftCond}). 
This is done using the constructions of Section \ref{subsec:driftCondFromLE}, which relies on conditions \ref{ass:reg}, \ref{ass:jacobbd}, \ref{ass:altUp}, \ref{ass:UEP}, \ref{ass:LE} and \ref{ass:UEPhat}, as well as topological irreducibility for $\widehat P$ (c.f. Corollary \ref{cor:propsPsiQ}\ref{cor:propsPsiQ:item2}). 
  \end{enumerate}

\medskip

\noindent Finally, if  \ref{ass:UEP2} holds, then Proposition \ref{prop:twoPointToMixing} 
applies, noting that $V$ as in \eqref{eq:formOfVChecklist} is automatically integrable with respect to $\pi^{(2)}$ on $M^{(2)}$ if $p < 1$ and $\pi = \Leb$ as assumed in \ref{ass:altUp}.  We conclude almost-sure exponential mixing for $f_\omega$ as in \eqref{eq:almostSureMixing4}. 

\section{Application to Pierrehumbert}\label{sec:pierrehumbert}
%
%
%
%
%
%
%

This section concerns the Pierrehumbert model presented in Section \ref{sub:PierreModel}. 
Section \ref{sub:topirr} establishes topological irreducibility for the one-point, projective and two-point Markov chains. Section \ref{sub:pierreLyap} treats positivity of the Lyapunov exponent for the Pierrehumbert model. Section \ref{sub:TchainAper} gives conditions for the existence of open small sets and aperiodicity for the three Markov chains. 
Section \ref{subsec:wrapUp} summarizes the proof of Theorem \ref{thm:main} on exponential mixing for the Pierrehumbert model.

\subsubsection*{Notation} 
Throughout, $\T^2$ is parametrized by $[0,2 \pi)^2$. 
We slightly abuse notation and write $x + y$ for the sum of $x, y \in \T^2$ modulo $2 \pi$ in both coordinates. For $x = (x_1, x_2) \in \T^2$ we write $[x]_i = x_i \in [0, 2 \pi)$ for the $i$-th coordinate of $x$, regarded as a real number modulo $2 \pi$.

The real number $\tau > 0$ is fixed, and $\Omega_0 = [0,2 \pi]^2$ denotes the space of possible noise parameters.  
Given $\omega = (\omega^1, \omega^2) \in \Omega_0$, we define
\begin{gather}
f_\omega = f^V_{\omega^2} \circ f^H_{\omega^1} \, , \\
f^H_\beta(x) = \begin{pmatrix} x_1 + \tau \sin(x_2 - \beta) \\ x_2 \end{pmatrix}  , \qquad 
f^V_\beta(x) = \begin{pmatrix} x_1 \\ x_2 + \tau \sin(x_1 - \beta) \end{pmatrix} ,
\end{gather}
for $x = (x_1, x_2) \in \T^2$. Given $\uo  = (\omega_1, \omega_2, \cdots) \in \Omega := \Omega_0^\N$, we write $f^n_\uo = f_{\omega_n} \circ \cdots \circ f_{\omega_1}$ for the corresponding composition of maps. At times, if $\omega_1, \cdots, \omega_n$ have been specified, we will abuse notation and write $x_n = f_{\omega_n} \circ \cdots \circ f_{\omega_1}(x)$.

\subsection{Topological irreducibility}\label{sub:topirr}

\subsubsection{Irreducibility of the one-point Markov chain} \label{sec:onepoint}

The random dynamical system $f^n_\uo$ corresponding to the one-point process on $\T^2$ is exactly controllable, as we show below. 
\begin{lem}[Exact controllability of the one-point chain]
Given $x,y \in  \T^2$, there exists $N=N(\tau) \in \N$ and $\uo = (\omega_1, \cdots, \omega_N) \in \Omega_0^N$ such that
 \begin{equation*}
  f^{N}_{\underline \omega}(x) =y \, .
 \end{equation*}
 \label{lem:1pointcontrollability}
 In particular, the one-point Markov kernel $P$ is topologically irreducible.
\end{lem}

\begin{proof}
To start, set $N = N(\tau) = \lceil \frac{4 \pi}{\tau} \rceil$ and observe that
\begin{align*}
\bar \tau_i := \frac{[x]_i - [y]_i}{N} 
\end{align*}
belongs to $[-\tau, \tau]$ for $i = 1,2$. 
 We now define $\uo^N = (\omega_1, \cdots, \omega_N) \in \Omega_0^N$ inductively as follows: given $\uo^n := (\omega_1, \cdots, \omega_n)$, choose $\omega_{n + 1} = (\omega_{n + 1}^1, \omega_{n + 1}^2)$ such that
 \[
 \tau \sin\left( [f^n_{\uo^n} (x)]_2 - \omega_{n + 1}^1 \right) = \bar \tau_1 \, , 
 \quad  \tau \sin\left( [f_{\omega_{n + 1}^1}^H \circ f^n_{\uo^n} (x)]_1 - \omega_{n + 1}^2 \right) = \bar \tau_2 \, , 
 \]
noting that for each $1 \leq n \leq N$, we have
 \[
f^n_{\uo^n}(x)= x + n (\bar \tau_1, \bar \tau_2) \, . 
 \]
 By our definitions, this ensures $f^N_{\uo^N}(x) =y$, as desired.  
\end{proof}

\subsubsection{Irreducibility of the projective Markov chain}
We now turn attention to the linear cocycle associated to the Jacobians $\Ac^n_{\uo, x} := D_x f^n_\uo$, regarding $\Ac^n, n \geq 1$ as a mapping $\Omega \times \T^2 \to GL_n(\R)$. 
We define the projective dynamics $\widehat f_\omega : \T^2 \times S^1 \to\T^2 \times S^1, S^1 \subset \R^2$ the unit circle, by
\[
\widehat f_\omega(x, v) = \left( f_\omega(x) , \frac{D_x f_\omega(v)}{|D_x f_\omega(v)|} \right) \,, 
\]
and write $\widehat f^n_{\uo} = \widehat f_{\omega_n} \circ \cdots \circ \widehat f_{\omega_1}$ for the time$-n$ composition. 
As we will show, the RDS $\widehat f^n_\uo$ is approximately controllable.
\begin{prop}\label{prop:projIrred}
For any $\eps > 0$ and $(x, v), (x_\star, v_\star) \in \T^2$, there exists $N \in \N$ and $\uo^N$ such that 
\[
d(\widehat f^N_{\uo^N}(x, v), (x_\star, v_\star))  < \eps \,. 
\]
Moreover, $N$ is bounded uniformly from above in terms of $\eps$ and $\tau$ alone. In particular, the Markov kernel $\widehat P$ associated to $\widehat f^n_\uo$ is topologically irreducible. 
\end{prop}

We begin the proof with the following straightforward computation and some notation: the 
individual time-one Jacobians $\Ac_{\omega, x} := D_x f_\omega$ are of the form
\begin{align}\label{eq:formOfLinearizationPH5}
\Ac_{\omega, x}  = \begin{pmatrix} 1 & C^H \\ C^V & 1 + C^H C^V \end{pmatrix} \, , 
\end{align}
where
\begin{align}
C^H = C^H(\omega, x) &= \tau \cos([x]_2 - \omega^1) \, ,  \\
C^V= C^V(\omega, x) &= \tau \cos \left( [f^H_{\omega^1}(x)]_1 - \omega^2 \right) \, . 
\end{align}
Given $x \in \T^2$, $v,w \in S^1$ and assuming $\omega_1, \cdots, \omega_n$ have been specified, let us write
\[
x_n := f^n(x) \, , \quad w_n := D_x f^n(w) \, , \quad v_n := \frac{D_x f^n(v)}{| D_x f^n(v)|} \, , 
\]
so that in particular $w_n$ and $v_n$ are parallel for all $n$, with $v_n \cdot w_n > 0$, while $(x_n, v_n) = \widehat f_{\omega_n} \circ \cdots \circ \widehat f_{\omega_1}(x, v)$. 
Lastly, given $u \in \R^2$ we write $[u]_i, i = 1,2$ for the $i$-th coordinate of $u$. 

Next, we establish the following intermediate Lemma allowing to approximately control to points of the form $(\bar x, (1/\sqrt2,1/\sqrt2)) \in \T^2 \times S^1$. 

\begin{lem}
Let $(x, v) \in \T^2 \times S^1$, $\bar x \in \T^2$, and $\eps > 0$. Then, there exists $N_1 \in \N$ and $\uo^{N_1}$ such that 
\[
d_{\T^2 \times S^1}\left( (x_{N_1}, v_{N_1}) , (\bar x, (1/\sqrt{2},1/\sqrt{2})) \right) < \eps \, .
\]
\label{lem:projpart}
\end{lem}

\begin{proof}[Proof of Lemma \ref{lem:projpart}]
We divide the proof into two steps: 
\begin{itemize}
	\item[(1)] \emph{Achieving correct angle. }There exists $N' \in \N$ and $\omega_1, \cdots, \omega_{N'} \in \Omega_0$ such that 
	$v_{N'} = (1/\sqrt{2}, 1/\sqrt{2})$. 
	\item[(2)] \emph{Rigid motions. }There exists $N'' \in \N$ and $\omega_{N' + 1}, \cdots, \omega_{N_1} \in \Omega_0 , N_1 := N' + N''$, such that $v_{N_1} = (1/\sqrt2,1/\sqrt2)$ and $d_{\T^2}(x_{N_1}, \bar x) < \eps$. 
\end{itemize}
We refer the reader to schematic in Figure~\ref{fig:proj} where we have provided a visual explanation of the proof of approximate controllability for the projective chain.
\begin{figure}
\centering
\resizebox{5.1cm}{5.1cm}{%
\begin{tikzpicture}
\draw[black,thick] (0,0) -- (2*3.1415,0) -- (2*3.1415,2*3.1415) -- (0,2*3.1415) -- (0,0);
\draw[blue,fill] (2,2) circle (.05cm) node[black,anchor=north west] {\footnotesize$x$} ;
\draw[blue,->,>=stealth, rotate around={70:(2,2)}] (2,2) -- (2.5,2)  node [below right,black] {\footnotesize$v$};
\draw[blue,fill] (4,4) circle (.05cm) node[black,anchor=north west] {\footnotesize$x_{N'}$} ;
\draw[blue,->,>=stealth, rotate around={45:(4,4)}] (4,4) -- (4.5,4)  node [below right,black] {\footnotesize$\bar v$};
\draw[red,fill] (3,1) circle (.05cm) node[black,anchor=north west] {\footnotesize$\bar x$} ;
\draw[red,->,>=stealth, rotate around={45:(3,1)}] (3,1) -- (3.5,1)  node [below right,black] {\footnotesize$\bar v$};
\draw [->,>=stealth,blue,thick](2,2) [out=-10,in=-100] to (4,3.95);
\end{tikzpicture}
}
\hspace{0.5cm}%
\resizebox{5.1cm}{5.1cm}{%
\begin{tikzpicture}
\draw[black,thick] (0,0) -- (2*3.1415,0) -- (2*3.1415,2*3.1415) -- (0,2*3.1415) -- (0,0);
\draw[blue,fill] (4,4) circle (.05cm) node[black,anchor=north west] {\footnotesize$x_{N'}$} ;
\draw[blue,->,>=stealth, rotate around={45:(4,4)}] (4,4) -- (4.5,4)  node [below right,black] {\footnotesize$\bar v$};
\draw[blue,fill] (3,1.4) circle (.05cm) ;
\draw[blue,->,>=stealth, rotate around={45:(3,1.4)}] (3,1.4) -- (3.5,1.4)  ;
\draw[red,fill] (3,1) circle (.05cm) node[black,anchor=north west] {\footnotesize$\bar x$} ;
\draw[red,->,>=stealth, rotate around={45:(3,1)}] (3,1) -- (3.5,1)  node [below right,black] {\footnotesize$\bar v$};
\draw [->,>=stealth,blue,thick](4,4) [out=100,in=100] to (3,1.42);
\draw[red!20] (3,1) circle (0.5cm) ;
 \draw[red!20,|->|,>=stealth, rotate around={180:(3,1)}] (3,1) -- (3.5,1)  node [below right,black] {\footnotesize$\eps$};
\end{tikzpicture}
}
\caption{A visual explanation of the proof of approximate controllability for the projective chain: we first achieve the correct alignment in projective space, i.e. $\bar v= (1/\sqrt{2},1/\sqrt{2})$ after which  we perform rigid motions to  bring the point within an $\eps$-neighborhood of the target point $\bar x$.}
\label{fig:proj}
\end{figure}
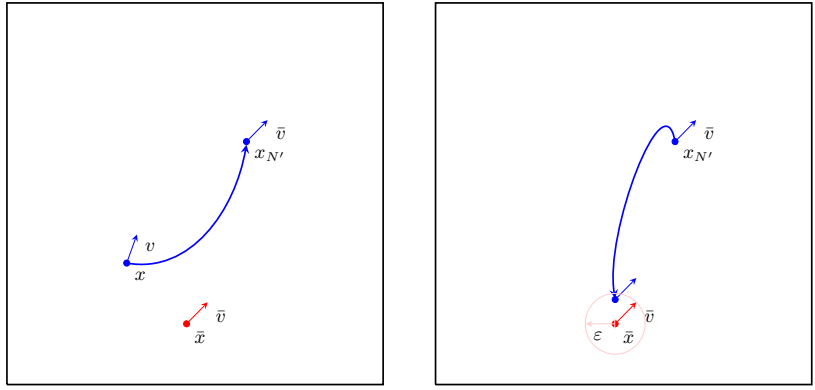

\smallskip
\noindent {\bf Step (1). }
Since $v$ is a unit vector, we must have that $[v]_i \geq 1/\sqrt{2}$ for one of $i = 1,2$. For now, let us assume that $[v]_2 \geq 1/\sqrt{2}$; if not, the same proof works with minor modifications, details omitted. 

Observe that for any $y\in \T^2$, by varying $\eta = (\eta^1, \eta^2) \in \Omega_0$ we can specify $(C^H(\eta, y), C^V(\eta, y))$ to take any value in $[-\tau, \tau]^2$. In particular, if $C^V(\eta, y) = 0$, then
\[
D_y f_\eta (u) = \begin{pmatrix} 1 & C^H(\eta, y) \\ 0 & 1 \end{pmatrix} u
= \begin{pmatrix} [u]_1 + C^H(\eta, y) [u]_2 \\ [u]_2 \end{pmatrix}
\]
for any $u \in \R^2$. With this in mind, we will choose $\omega_i = (\omega_i^1, \omega_i^2)$ for each $i$ so that
\[
C^V(\omega_i, x_{i-1}) = 0 \, , \quad C^H(\omega_i, x_{i-1}) = \frac{[w]_2 - [w]_1}{10} \, . 
\]
This ensures that at time $N' = 10$, we have that
\[
w_{N'} = \begin{pmatrix} [w]_1 + [w]_2 \sum_{i = 1}^{N'} C^H(\omega_i, x_{i-1})  \\ [w]_2 \end{pmatrix} = \begin{pmatrix} [w]_2 \\ [w]_2 \end{pmatrix} \, .
\]
In particular, $v_{N'} = (1/\sqrt2,1/\sqrt2)$, as desired. 

\smallskip
\noindent {\bf Step (2). } We will choose $\omega_{i}, i \geq N' + 1$ subject to two constraints. The first is that the relation
\[
C^H(\omega_i, x_{i-1}) = \frac{C^V(\omega_i, x_{i-1})}{1 - C^V(\omega_i, x_{i-1})}
\]
holds for each $i \geq N' + 1$, and ensures (by \eqref{eq:formOfLinearizationPH5}) that $v_{i} = v_{N'} = (1/\sqrt{2}, 1/\sqrt{2})$ for all $i \geq N' + 1$. 

For the second constraint, let $\zeta = (\zeta^1, \zeta^2) \in [0, \tau/2 \pi]^2$, to be specified momentarily. Our second constraint will be that for all $i \geq N' + 1$, we have that
\[
C^H(\omega_i, x_{i-1}) = \sqrt{\tau^2 - (2 \pi \zeta^1 )^2}\, , \quad C^V(\omega_i, x_{i-1}) = \sqrt{\tau^2 - (2 \pi \zeta^2)^2} \, .
\]
That both the first and second constraints are satisfied requires that $\zeta^1, \zeta^2$ satisfy the relation
\begin{align}\label{eq:conditionZeta5}
\sqrt{\tau^2 - (2 \pi \zeta^1)^2} 
= 
\frac{\sqrt{\tau^2 - (2 \pi \zeta^2)^2} }{1 - \sqrt{\tau^2 - (2 \pi \zeta^2)^2} } \, .
\end{align}
With the first and second constraints in place, observe that for $i \geq N' + 1$ we will have that
\[
(x_i, v_i) = \left( x_{i-1} + \begin{pmatrix} \sqrt{\tau^2 - (C^H(\omega_i, x_{i-1}))^2} \\ \sqrt{\tau^2 - (C^V(\omega_i, x_{i-1}))^2} \end{pmatrix} ,v_{i-1}\right) = \left( x_{i-1} + \begin{pmatrix} 2 \pi \zeta^1 \\ 2 \pi \zeta^2 \end{pmatrix} , v_{i-1} \right) \, .
\]

\begin{cla}
There exists $(\zeta^1, \zeta^2) \in [0, \tau / (2 \pi)]^2$ such that \eqref{eq:conditionZeta5} holds, and moreover, $\zeta^1$ and $\zeta^2$ are \emph{rationally independent}: there are no nonzero solutions $a_1, a_2 \in \mathbb Q, b \in \mathbb Z$ to the equation $a_1 \zeta^1 + a_2 \zeta^2 = b$. 
\end{cla}

Assuming the Claim and fixing $\zeta^1, \zeta^2$, the Weyl Equidistribution Theorem in dimension $d = 2$ (see Corollary \ref{cor:weyl} in Appendix \ref{app:Weyl}) implies there exists $N'' \in \N$, uniformly bounded from above in terms of $\eps$, such that 
\[
x_{N''} = x_{N'} + N''\begin{pmatrix} 2 \pi \zeta^1 \\ 2 \pi \zeta^2\end{pmatrix}
\]
is within distance $\eps$ of $\bar x$. 
\end{proof}

\begin{proof}[Proof of Claim]
It can be checked using standard computer algebra software, e.g., Mathematica, that
the set of solutions $\Sc = \{ (\zeta^1, \zeta^2)\}$ in $[0,\tau / 2 \pi]^2$ to the identity \eqref{eq:conditionZeta5} is of the form $\Sc = \{ (x, g(x)) : x \in [0,\tau / 2 \pi]\}$ where $g = g_\tau : [0, \tau / 2 \pi] \to [0, \tau / 2 \pi]$ is of the form
\[
g(x) = \frac{\sqrt{\tau ^6-3 \tau ^4+2 \tau ^2 \sqrt{\tau ^2-4 \pi ^2 x^2}+16 \pi ^4 \left(\tau ^2-1\right) x^4+4 \pi ^2 x^2 \left(-2 \tau ^4+4 \tau ^2-2 \sqrt{\tau ^2-4 \pi ^2 x^2}+1\right)}}{2 \pi  \left(\tau ^2-4 \pi ^2 x^2-1\right)^2}
\]
By inspection, $g$ is real-analytic on the open interval $(0, \tau / 2 \pi)$. 

The set $\mathcal{RD}$ of rationally dependent $\zeta$ is of the form
\[
\mathcal{RD} = \left( \cup_{a \in \mathbb Q} \{ (a,y) : y \in \R\} \cup \bigcup_{a_1, a_2 \in \mathbb Q} \{ (x, a_1 x + a_2) : x \in \R\} \right)  \cap [0,\tau/2 \pi]^2 \, , 
\]
and lies on a countable union of affine lines. The intersection of any real-analytic graph and an affine line has at finitely many points, and so $\mathcal S \cap \mathcal {RD}$ is at-most countable. Since $\mathcal S$ is uncountable, the proof is complete. 
\end{proof}

\begin{proof}[Proof of Proposition \ref{prop:projIrred}]
Let $\eps > 0$ and $(x, v), (x_\star, v_\star) \in \T^2$ be fixed. Using again the fact that $f_\omega^{-1} = f^H_{\omega^1 + \pi} \circ f^V_{\omega^2 + \pi}$, a slight modification of Step 1 in the proof of Lemma \ref{lem:projpart} implies that with $N' = 10$, we have that there exists $\hat \omega^{N'} \in \Omega_0^{N'}$ such that
\[
D_{x_\star} (f^{N'}_{\hat \uo^{N'}})^{-1} (v_\star)
\]
is parallel to $(1/\sqrt{2}, 1/\sqrt{2})$. Define $\bar x := (f^{N'}_{\hat \uo^{N'}})^{-1} (x_\star)$. 
Let $L_0$ be an $\uo$-uniform upper bound for $\Lip(\widehat f^{N'}_\uo)$. 
Set $\eps_0 = \eps / 2 L_0$ and let $N_1 \in \N, \uo^{N_1}$ be as in 
Lemma \ref{lem:projpart} so that $d_{\T^2}(x_{N_1} , \bar x) < \eps_0$ and 
$v_{N_1} = (1/\sqrt2,1/\sqrt2)$. Setting $\omega_{N_1 + i} = \hat \omega_i, 1 \leq i \leq N', N := N_1 + N'$, 
we conclude that
\[
d_{\T^2 \times S^1}((x_N, v_N), (x_\star, v_\star)) < \eps \, , 
\]
as desired. 
\end{proof}

\subsubsection{Irreducibility of the two-point Markov chain}

 Recall that $\Delta \subset \T^2 \times \T^2$ denotes the diagonal, $\Delta = \{ (x, x) : x \in \T^2 \}$. 

\begin{prop}\label{prop:twoPointIrred}
Given $(x,y), (x_\star, y_\star) \in \T^2 \times \T^2 \setminus \Delta$ and $\eps >0$, there exists $N \in \N$ and some $\uo^N \in \Omega_0^N$ such that
\[
\dist( f^N_{\uo^N} (x), x_\star) + \dist( f^N_{\uo^N}(y) , y_\star ) < \eps \,. 
\]
In particular, the two-point Markov kernel $P^{(2)}$ is topologically irreducible. 
\end{prop}

To prove Proposition \ref{prop:twoPointIrred}, we start with following preliminary computation.

\begin{lem}
Given  $x,y \in \T^2, \gamma^i \in [0,2 \pi), i = 1,2$, define
\[
\beta^1 = \frac{[x]_2 + [y]_2}{2} - \gamma^1 \, , \quad \beta^2 = \frac{[x]_1 + [y]_1}{2} - \gamma^2 \, .
\]
Then,
\begin{align}
[f^H_{\beta^1}(x) - f^H_{\beta^1}(y)]_1 &= [x - y]_1 + 2 \tau \sin \left( \frac{[x -y]_2}{2} \right) \cos (\gamma^1) 
\end{align}
and
\begin{align}
[f^V_{\beta^2}(x) - f^V_{\beta^2}(y)]_2 &= [x - y]_2 + 2 \tau \sin \left( \frac{[x -y]_1}{2} \right) \cos (\gamma^2)
\end{align}
In particular, setting $\gamma^i = \frac{\pi}{2}$, one has that 
$[f^H_{\beta^1}(x) - f^H_{\beta^1}(y)]_1 = [x - y]_1$ and 
$[f^V_{\beta^2}(x) - f^V_{\beta^2}(y)]_2 = [x - y]_2$. 
\label{lem:rigidbody}
\end{lem} 

The following intermediate Lemma establishes approximate controllability to 
configurations in $\T^2 \times \T^2 \setminus \Delta$ of a special form. 

\begin{lem}\label{lem:twoPtAccess}
For any $x, y \in \T^2 \times \T^2 \setminus \Delta, \delta_1 > 0, \eps > 0$ and $\bar x \in \T^2$ such that $\frac{\tau \cos(\delta_1/2)}{2 \pi}$ is irrational, there exists $N \in \N$ such that $x_N = f^N_{\uo^N}(x), y_N = f^N_{\uo^N}(y)$ satisfy
\begin{align}\label{eq:lem54}
x_N - y_N = (\delta_1, 0) \quad \text{ and } d_{\T^2}(x_N, \bar x) < \eps \,. 
\end{align}
Moreover, we have that $N \leq \hat N$, where $\hat N = \hat N(d_{\T^2}(x, y), \delta_1, \eps, \tau)$; in particular, $N$ is bounded uniformly over $(x,y) \in \T^2 \setminus \Delta$ with $d_{\T^2}(x,y) \geq \eta$ for each fixed $\eta > 0$. 
\end{lem}

\begin{proof}
Throughout, when $\uo^n \in \Omega_0^n$ has been specified, we will write $x_n = f^n_{\uo^n}(x), y_n = f^n_{\uo^n}(y)$. We split the proof into two steps: 
\begin{itemize}
\item[(1)] {\it Achieving correct separation.} There exist $N_1 \in \N$ and $\uo^{N_1}$ 
such that $[x_{N_1} - y_{N_1}]_1 = \delta_1$
and $[x_{N_1}]_2 = [y_{N_1}]_2$. 

\item[(2)] {\it Rigid motions. } There exist $N_2 \in \N$ and $\uo^{N_2}$ such that 
$x_{N_1 + N_2} - y_{N_1 + N_2} = x_{N_1} - y_{N_1} = (\delta_1, 0)$ and $\dist_{\T^2}(x_{N_1 + N_2}, \bar x) < \eps$. 
\end{itemize}
As we will see, the value $N_1$ will be bounded uniformly from above depending only on $\dist_{\T^2}(x, y)$, $\tau$ and $\delta_1$, while $N_2$ will depend only on $\delta_1, \tau$ and $\eps$. We refer the reader to schematic in Figure~\ref{fig:2point} where we have provided a visual explanation of the proof of approximate controllability.

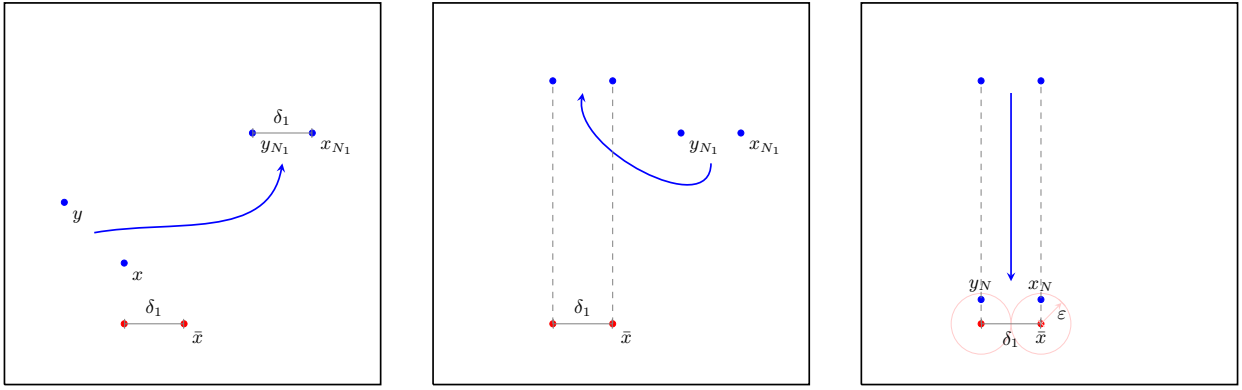
\begin{figure}
\centering
\resizebox{5.1cm}{5.1cm}{%
\begin{tikzpicture}
\draw[black,thick] (0,0) -- (2*3.1415,0) -- (2*3.1415,2*3.1415) -- (0,2*3.1415) -- (0,0);
\draw[blue,fill] (2,2) circle (.05cm) node[black,anchor=north west] {\footnotesize$x$} ;
\draw[blue,fill] (1,3) circle (.05cm) node[black,anchor=north west] {\footnotesize$y$} ;
\draw[blue,fill] (3.1415+1,3.1415+1) circle (.05cm) node[black,anchor=north west] {\footnotesize$y_{N_1}$} ;
\draw[blue,fill] (3.1415+2,3.1415+1) circle (.05cm) node[black,anchor=north west] {\footnotesize$x_{N_1}$} ;
\draw[red,fill] (2,1) circle (.05cm);
\draw[red,fill] (3,1) circle (.05cm) node[black,anchor=north west] {\footnotesize$\bar x$} ;

\draw [|-|,gray] (3.1415+1,3.1415+1) -- (3.1415+2,3.1415+1) node[midway,above,black] {\footnotesize$\delta_1$};
\draw [|-|,gray] (2,1) -- (3,1) node[midway,above,black] {\footnotesize$\delta_1$};
\draw [->,>=stealth,blue,thick](1.5,2.5) [out=10,in=-100] to (3.1415+1.5,3.1415+0.5);
\end{tikzpicture}
}
\hspace{0.5cm}%
\resizebox{5.1cm}{5.1cm}{%
\begin{tikzpicture}
\draw[black,thick] (0,0) -- (2*3.1415,0) -- (2*3.1415,2*3.1415) -- (0,2*3.1415) -- (0,0);
\draw[blue,fill] (3.1415+1,3.1415+1) circle (.05cm) node[black,anchor=north west] {\footnotesize$y_{N_1}$} ;
\draw[blue,fill] (3.1415+2,3.1415+1) circle (.05cm) node[black,anchor=north west] {\footnotesize$x_{N_1}$} ;
\draw[blue,fill] (2,5) circle (.05cm);
\draw[blue,fill] (3,5) circle (.05cm);
\draw[red,fill] (2,1) circle (.05cm);
\draw[red,fill] (3,1) circle (.05cm) node[black,anchor=north west] {\footnotesize$\bar x$} ;
\draw [|-|,gray] (2,1) -- (3,1) node[midway,above,black] {\footnotesize$\delta_1$};
\draw [->,>=stealth,blue,thick](3.1415+1.5,3.1415+0.5) [out=-90,in=-100] to (2.5,4.8);
\draw [-,dashed,gray] (2,1) -- (2,5);
\draw [-,dashed,gray] (3,1) -- (3,5); 
\end{tikzpicture}
}
\hspace{0.5cm}%
\resizebox{5.1cm}{5.1cm}{%
\begin{tikzpicture}
\draw[black,thick] (0,0) -- (2*3.1415,0) -- (2*3.1415,2*3.1415) -- (0,2*3.1415) -- (0,0);
\draw[blue,fill] (2,5) circle (.05cm);
\draw[blue,fill] (3,5) circle (.05cm);
\draw[red,fill] (2,1) circle (.05cm);
\draw[red,fill] (3,1) circle (.05cm) node[black,anchor=north] {\footnotesize$\bar x$} ;
\draw[blue,fill] (2,1.4) circle (.05cm) node[black,anchor=south] {\footnotesize$ y_{N}$};
\draw[blue,fill] (3,1.4) circle (.05cm) node[black,anchor=south] {\footnotesize$ x_{N}$} ;
\draw[red!20] (2,1) circle (0.5cm);
\draw[red!20] (3,1) circle (0.5cm) ;
\draw [|-|,gray] (2,1) -- (3,1) node[midway,below,black] {\footnotesize$\delta_1$};
 \draw [->,>=stealth,blue,thick](2.5,4.8)  to (2.5,1.7);
 \draw[red!20,|->|,>=stealth, rotate around={45:(3,1)}] (3,1) -- (3.5,1)  node [below,black] {\footnotesize$\eps$};
 \draw [-,dashed,gray] (2,1) -- (2,5);
\draw [-,dashed,gray] (3,1) -- (3,5); 
\end{tikzpicture}
}
\caption{A visual explanation of the proof of approximate controllability for the two-point chain: we first achieve the correct separation between the two points after which  we perform rigid motions
to align the points with the same first-coordinate as the target point, and then bring the process within an $\eps$-neighborhood of the target points.}
\label{fig:2point}
\end{figure}
\smallskip

\noindent{\bf Step (1). } Assume $[x]_2 \neq [y]_2$; we will remove this constraint momentarily. To start, we will find an appropriate $N' = N'(\tau, [x-y]_2)$ such that $[x_{N'} - y_{N'}]_2 = [x - y]_2$ and $[x_{N'} - y_{N'}]_1 = \delta_1$. In each step, we are free to choose the phase shift $\omega_n^2$ so that $[x_{n + 1} - y_{n+1}]_2 = [x_n - y_n]_2 = \cdots = [x - y]_2$ (Lemma \ref{lem:rigidbody}), and so it suffices to specify only the sequence $\omega^1_1, \cdots, \omega^1_{N'}$. 

For this, set 
\[
N' = \left\lceil \frac{4 \pi }{2 \tau | \sin \left(\frac{[x - y]_2}{2}\right) | } \right\rceil
\]
and for each $1 \leq i \leq N'$, define $\omega_i^1 = \frac{[x_{i-1} + y_{i-1}]_2}{2} - \gamma$, where $\gamma$ is chosen so that
\[
2 \tau \sin \left(\frac{[x-y]_2}{2}\right) \cos \gamma = \frac{1}{N'} ( \delta_1 - [x - y]_1)
\]
With this assignment, it is straightforward to check that 
\[
[x_i - y_i]_1 = \frac{1}{N'} ( \delta_1 - [x - y]_1) + [x_{i-1} - y_{i-1}]_1 = \cdots = \frac{i}{N'} (\delta_1 - [x-y]_1) + [x-y]_1 \, ,
\]
hence $[x_{N'} - y_{N'}]_1 = \delta_1$. 

Now, we will define phase shifts $\omega_{N' + 1}, \cdots, \omega_{N' + N''}$, $N_1 := N' + N''$ so that $[x_{N_1} - y_{N_1}]_1 = \delta_1$ and $[x_{N_1}]_2 = [y_{N_1}]_2$, completing Step (1). Using Lemma \ref{lem:rigidbody} as earlier, we can always find $\omega_{N'+1}^1, \cdots, \omega_{N' + N''}^1$ so that $[x_{N_1} - y_{N_1}]_1 = \cdots = [x_{N'} - y_{N'}]_1 = \delta_1$; thus it suffices to specify $\omega_{N' + 1}^2, \cdots, \omega_{N' + N''}^2$ to achieve the desired vertical displacement. Since $[x_{N'} - y_{N'}]_1 = \delta_1 \neq 0$, we can do this by repeating the above argument on exchanging the roles played by the horizontal and vertical axes, noting that the resulting value of $N''$ depends only on the horizontal separation $\delta_1$ and $\tau$; further details are omitted. 

If $[x - y]_2 = 0$, then $[x - y]_1 \neq 0$ must hold, and it is straightforward to check that $\hat x = f_{(0,0)} (x), \hat y = f_{(0,0)}(y)$ have the property that $[\hat x - \hat y]_2 \neq 0$, with $[\hat x - \hat y]_2$ bounded uniformly away from 0 depending only on $[x-y]_1$ and $\tau$. One can now apply the preceding argument with $(\hat x, \hat y)$ replacing $(x,y)$. 

\smallskip 
\noindent {\bf Step (2). } To start, we will specify $M' \in \N$ depending only on $\tau$ and $\omega_{N_1 +1}, \cdots, \omega_{N_1 + M'}$ so that 
$x_{N_1 + M'} - y_{N_1 + M'} = x_{N_1} - y_{N_1} = (\delta_1, 0)$ and 
$[x_{N_1 + M'}]_1 = [\bar x]_1$. Again using Lemma \ref{lem:rigidbody}, we can always choose $\omega_{N_1 + 1}^2, \cdots, \omega_{N_1 + M'}^2$ to leave $[x_i - y_i]_2$ unchanged for $i = N_1, N_1 + 1, \cdots, N_1 + M'$, and so it suffices to specify $\omega_{N_1 + 1}^1, \cdots, \omega_{N_1 + M'}^1$. To do this, we can repeat the construction in the proof of exact controllability of the 1-point process (Lemma \ref{lem:1pointcontrollability}): set $M' = \lceil \frac{4 \pi}{\tau} \rceil$ and for each $1 \leq i \leq M'$, chose $\omega_{N_1 + i}^1$ so that
\[
\tau \sin ( [x_{N_1 + i - 1}]_2 - \omega_{N_1 + i}^1) = \frac{ [\bar x]_1 - [x_{N_1}]_1 }{M'} \, . 
\]
As before, $[x_{N_1 + i}]_1 = [x_{N_1}]_1 + \frac{i}{M'} ( [\bar x]_1-[x_{N_1}]_1 )$, and so $[x_{N_1 + M'}]_1 = [\bar x]_1$ as desired. Since $[x_{N_1 + i}]_2 = [y_{N_1 + i}]_2$ by construction for each $1 \leq i \leq M'$, it follows\footnote{Along each horizontal line, the horizontal shear $f^H_\beta$ acts rigidly for all $\beta$. That is,   if $[x]_2 = [y]_2$, then $[f^H_\beta(x) - f^H_\beta(y)]_1 = [x - y]_1 = a - c$.} that $[x_{N_1 + M'} - y_{N_1 + M'}]_1 = [x_{N_1} - y_{N_1}]_1 = \delta_1$, hence $[y_{N_1 + M'}]_1 = [\bar x]_1 + \delta_1$. 

Now, given $\eps > 0$, we will find $M'' = M''(\tau, \delta_1, \eps) \in \N$ and $\omega_{N_1 + M' + 1}, \cdots, \omega_{N_1 + N_2}$, $N_2 := M' + M''$, such that $\dist_{\T^1} ( [x_{N_1 + N_2}]_2 , [\bar x]_2) < \eps$ while preserving the constraints $x_{N_1 + M' + i} - y_{N_1 + M' + i} = (\delta_1, 0)$ and $[x_{N_1 + M' + i}]_1 = [\bar x]_1$ for all $1 \leq i \leq M''$. To do this, at each $1 \leq i \leq M''$, the choice 
\[
\omega_{N_1 + M' + i}^1 = [x_{N_1 + M' + i-1}]_2 = [y_{N_1 + M' + i-1}]_2
\]
will ensure $[x_{N_1 + M' + i}]_1 = [x_{N_1 + M' + i-1}]_1 = \cdots = [\bar x]_1$ and $[y_{N_1 + M' + i}]_1 = [\bar x]_1 + \delta_1$. Meanwhile, 
\[
\omega_{N_1 + M' + i}^2 = \frac{[x_{N_1 + M' + i-1}]_1 + [y_{N_1 + M' + i-1}]_1}{2} - \frac{\pi}{2} = \frac{2 [\bar x]_1 + \delta_1}{2} - \frac{\pi}{2}
\]
will ensure $[x_{N_1 + M' + i}]_2 = [y_{N_2 + M' + i}]_2$ (Lemma \ref{lem:rigidbody}); indeed, we have that 
\[
[x_{N_1 + M' + i}]_2 = [y_{N_2 + M' + i}]_2 = [x_{N_1 + M' + i-1}]_2 +  \tau \cos \frac{\delta_1}{2} \,. 
\]
By irrationality of $\frac{\tau \cos(\delta_1 / 2)}{2 \pi}$, the Weyl Equidistribution Theorem (see Corollary \ref{cor:weyl} in Appendix \ref{app:Weyl}) implies that there exists $M''$ (depending only on $\eps, \tau$ and $\delta_1$) such that 
\[
\dist_{\T^1} ( M'' \tau \cos(\delta_1 / 2), [\bar x]_2) < \frac{\eps}{2} \,. 
\]
Note that here, $\T^1$ is parametrized by $[0,2 \pi)$. This choice of $M''$ implies
\[
\dist_{\T^2} ( x_{N_1 + N_2}, \bar x) + \dist_{\T^2}(y_{N_1 + N_2}, \bar x + (\delta_1, 0)) < \eps \, , 
\]
as desired. 
\end{proof}

\begin{proof}[Proof of Proposition \ref{prop:twoPointIrred}]
Fix $\eps > 0$ and $(x, y), (x_\star, y_\star) \in \T^2\times \T^2 \setminus \Delta$. 
Let $\delta_1 > 0$ be such that $\frac{\tau \cos( \delta_1/2)}{2 \pi}$ is irrational. Observe that
\[
f_{(\beta^1, \beta^2)}^{-1} = f^H_{\beta^1 + \pi} \circ f^V_{\beta^2 + \pi} \,. 
\]
In particular, we can (with cosmetic changes) apply Step 1 in the proof of Lemma \ref{lem:twoPtAccess} to obtain $N_0 = N_0(\delta_1, \tau, d_{\T^2}(x_\star, y_\star))$ and $\hat{\uo}^{N_0}$ such that $\hat x := (f^{N_0}_{\hat{\uo}^{N_0}})^{-1}(x_\star)$ and $\hat y = 
(f^{N_0}_{\hat{\uo}^{N_0}})^{-1}(y_\star)$ satisfy
\[
\hat x - \hat y = (\delta_1, 0) \,. 
\]
Let $L_0$ be a uniform ($\eps$-independent) upper bound for $\Lip(f^{N_0}_\uo)$ and fix $\eps_0 = \eps / 2 L_0$. Let $N_0' \in \N, \uo^{N_0'}$ be as in Lemma \ref{lem:twoPtAccess} so that \eqref{eq:lem54} holds with the replacements $\bar x \mapsto \hat x, \eps \mapsto \eps_0, N \mapsto N_0'$. Define now $N = N_0 + N_0'$ and let
$\uo^N$ be the concatenation $\hat{\uo}^{N_0}\uo^{N_0'}$, i.e., 
\[
\uo^N = (\omega_1, \cdots, \omega_{N_0'}, \hat \omega_1, \hat \omega_2, \cdots, \hat \omega_{N_0}) \,,
\]
observing that by our construction, $d_{\T^2} (f^N_{\uo^N}(x), x_\star) + d_{\T^2}(f^N_{\uo^N}(y), y_\star) < \eps$, as desired. 
\end{proof}

\subsection{Positivity of Lyapunov exponents}\label{sub:pierreLyap}
The main goal is to verify the assumptions of Proposition \ref{prop:suffCondRuleOutFurst} for the Pierrehumbert mappings $(f^n_\uo)$. In this setting, it suffices to show that $\exists n \geq 1$ and $\uo^n_\star \in \Omega_0^n$ such that (i) the mapping
\[
\Phi_{x_\star} (\uo^n) := f^n_{\uo^n} (x_\star) \quad \Phi_{x_\star} : \Omega_0^n \to \T^2
\]
is a submersion at $\uo^n_\star$; and (ii) that the mapping
\[
\widehat \Phi_{x_\star} (\uo^n) := D_{x_\star} f^n_{\uo^n} \, , \quad \widehat \Phi_{x_\star} : \Omega_0^n \to SL_2(\R)
\]
has the property that $D_{\uo^n_\star} \widehat \Phi_{x_\star}$ maps 
$\Sigma_{*} := \ker D_{\uo^n_\star} \Phi_{x_\star}$ surjectively onto $T_{\widehat \Phi_{x_\star}(\uo_\star^n)} SL_2(\R)$. 

We note that if property (i) holds at some $(\uo^n_\star, x_\star)$, observe that $\Sigma_\star$ 
must be at most $2n - 2$ dimensional, since each time increment injects two new real noise parameters. As the target space for $D_{\uo_\star^n} \widehat \Phi_{x_\star}$ is three-dimensional ($SL_2(\R)$ is 3-dimensional), property (ii) enforces the constraint $2n - 2 \geq 3$, hence $n \geq 3$ since $n$ is an integer. 

We have verified properties (i) and (ii) directly for $n = 3$ at the following values of 
$x_\star$ and $\uo_\star^3$: 
\begin{align}
x_\star=(\pi/2,\pi),\qquad \uo^3_\star=(  \omega^1_\star, \omega^2_\star, \omega^3_\star)=((0,0),(3\pi/2 + 1,\pi/2 - 1),(3\pi/2 + 1 ,5\pi/2 - 2)).
\end{align}
The computation itself is lengthy and only summarized briefly below. To start, we have
\begin{align}
D_{\uo^{3}_\star} \Phi_{x_\star} =
\begin{pmatrix}
1 & 0 & 0 & 0 & 0 & 0\\
2 & 0 & 0 & -1 &0 &-1
\end{pmatrix} \, . 
\end{align}
In particular, property (i) is satisfied and 
$\Sigma_\star = \ker D_{\uo^3_\star} \Phi_{x_\star}$
 is 4-dimensional and spanned by the columns of the matrix
\begin{align}
K=
\begin{pmatrix}
0 & 0 & 0 & 0 \\
0 & 0 & 0 & 1\\
0 & 0 & 1 & 0\\
 -1 & 0 & 0 & 0 \\
 0 & 1 & 0 & 0\\
 1 & 0 & 0 & 0
\end{pmatrix}
\end{align}
At this $x_\star, \uo^3_\star$, we have moreover that
\begin{align}
M := D_{\uo_\star^{3}} \widehat \Phi_{x_\star}=
\begin{pmatrix}
1 & 0 & 0 & -1 & -1 & 0 \\
0 & 0 & -1 & 0 & 0 & 0\\
0 & 1 & 0 & -1 & -1 & 0\\
 1 & -1 & -2 & 0 & 0 & 0 
\end{pmatrix} \, , 
\end{align}
having identified the space of $2\times2$ real matrices with $\R^4$ via the parametrization 
\[
\begin{pmatrix} a & b \\ c & d \end{pmatrix} \mapsto \begin{pmatrix} a \\ b \\ c \\ d \end{pmatrix}
\]
Now, to prove $M|_{\Sigma_\star}$ surjects, it suffices to show its rank as a linear operator is at least $3$; for this, it suffices to notice that 
\begin{align}
M K =
\begin{pmatrix}
1 & -1 & 0 & 0 \\
0 & 0 & -1 & 0\\
1 & -1 & 0 & 1\\
0 & 0 & -2 & -1 
\end{pmatrix}
\end{align}
 has rank 3. Thus, (i) and (ii) hold at this choice of $(x_\star, \uo^3_\star)$, and Proposition \ref{prop:suffCondRuleOutFurst} applies. 
 
\begin{prop}\label{prop:LyapPierre}
The dynamical system generated by the Pierrehumbert model \eqref{eq:PierreModel}-\eqref{eq:fomega} has a positive Lyapunov exponent $\lambda_1>0$.
\end{prop}

\begin{rmk}
In general, the number $n$ of iterates necessary to use the sufficient condition in Proposition \ref{prop:suffCondRuleOutFurst} depends on the dimension $d$ of the phase space $M$ and the number $\ell$ of degrees of freedom in the noise space $\Omega_0$ at each time increment. To work out the precise relationship: notice that 
$SL_d(\R)$ has dimension $d^2 - 1$, while $\Sigma_\star$ will be at most $\ell n - d$-dimensional, hence the inequality 
\begin{align}
n\geq (d^2+d-1)/\ell
\end{align}
must be satisfied. 
In particular, lowering the number of degrees of freedom $\ell$ in $\Omega_0$ and/or increasing the dimension $d$ of the phase space $M$ forces $n$ to rise, making this condition progressively harder to check. 
\end{rmk}

%

\subsection{Small set property and aperiodicity}\label{sub:TchainAper}
We now check the existence of open small sets and aperiodicity 
for each of the one-point, projective and two-point Markov chains using
Proposition \ref{prop:Tchain3} and Lemma \ref{lem:aperiodicSuffCond}. 


\subsubsection{The one-point process}\label{subsub:oneptsmallaper}
We take $n=1$ in Proposition \ref{prop:Tchain3}; it suffices to find $x_\star$ and $\omega_\star$ such that $\Phi_{x_\star} : \Omega_0 \to \T^2$ is a submersion at $\omega_\star$. We set
\begin{align}
x_\star=(0,0),\qquad \uo^1_\star=(0,0) \, , 
\end{align}
so that
\begin{align}
D_{\uo^{1}_\star} \Phi_{x_\star}=
\begin{pmatrix}
1 & 1 \\
1 & 2 
\end{pmatrix}
\end{align}
has full rank. Moreover, $x_\star$ is a fixed point of $f_{\omega_\star}$, so that Lemma \ref{lem:aperiodicSuffCond} implies aperiodicity, as desired. 

\subsubsection{The projective process}\label{subsub:projsmallaper}
We apply Proposition \ref{prop:Tchain3} with $n=2$ (notice that $n=1$ cannot work). We verify the submersion condition for $\Phi_{(x_\star, v_\star)}(\uo^2) = \widehat f_{\uo^2}(x_\star, v_\star)$ 
at $\uo^2_\star$ for
\begin{align}
x_\star=(0,0), \qquad v_\star=(1,0),\qquad \uo^2_\star=(  \omega^1_\star, \omega^2_\star)=((0,0),(\pi/2,\pi - 1)) \, .
\end{align}
The corresponding Jacobian is given by
\begin{align}
D_{\uo_\star^{2}} \Phi_{(x_\star,v_\star)} =
\begin{pmatrix}
-1 & 0 & 0 & 0 \\
0 & -1 & 0 & 1\\
0 & 0 & 0 & 0\\
1 & 1 & 1 & 0 
\end{pmatrix} \, ,
\end{align}
where here we treat $T_{(x, v)} \T^2 \times S^1$ for $(x, v) \in \T^2 \times S^1$
as a subspace of $\R^4$. 
As one can check, this matrix has rank 3 and therefore satisfies the submersion condition. 

For Lemma \ref{lem:aperiodicSuffCond}, we check directly that 
\begin{align}
x_\star=(0,0),  \qquad v_\star=\frac{1}{\sqrt{10}}\left(\sqrt{5-\sqrt{5}},\sqrt{5+\sqrt{5}}\right)
\end{align}
has the property that $\widehat f_{\omega_\star}(x_\star, v_\star) = (x_\star, v_\star)$. 

\subsubsection{The two-point process}\label{subsub:twoptsmallaper}
We apply Proposition \ref{prop:Tchain3} with $n=2$ and
\begin{align}
x_\star=(\pi,\pi), \qquad y_\star=(0,0),\qquad \uo^2_\star=(  \omega^1_\star, \omega^2_\star)=((0,0),(0,0)) \, . 
\end{align}
Identifying $T_{(x,y)} \T^2 \times \T^2 \cong \R^2 \times \R^2$ for $(x,y) \in \T^2 \times \T^2$, the corresponding Jacobian is
\begin{align}
D_{\uo^{2}_\star} \Phi_{(x_\star,y_\star)} = 
\begin{pmatrix}
2 & -1 & 1 & 0 \\
-3 & 2 & -1 & 1\\
-2 & -1 & -1 & 0\\
-3 & -2 & -1 & -1 
\end{pmatrix} \, .
\end{align}
This matrix is invertible, hence surjective, as desired. 
Checking aperiodicity with Lemma \ref{lem:aperiodicSuffCond}, 
the choice
\begin{align}
\omega_\star=(0,0) \, , \quad x_\star=(\pi,\pi),  \qquad y_\star=(0,0)
\end{align}
has the property that $f_{\omega_\star}(x_\star) = x_\star, f_{\omega_\star}(y_\star) = y_\star$.

\subsection{Almost-sure exponential mixing for the Pierrehumbert model}\label{subsec:wrapUp}

Let us now summarize the proof of Theorem \ref{thm:main} in terms of conditions \ref{suffcond0}--\ref{suffcond4} described in Section \ref{sec:checklist}.
\begin{itemize}
\item The basic assumptions in  \ref{suffcond0} are evident for the Pierrehumbert model $f^n_\uo$ as defined in \eqref{eq:PierreModel}. 
\item Uniform geometric ergodicity of the one-point kernel $P$ as in condition  \ref{suffcond1} (resp. projective kernel $\widehat P$ as in condition  \ref{suffcond3}) follows from Harris's Theorem (Theorem \ref{thm:Harris}), having checked topological irreducibility (Lemma \ref{lem:1pointcontrollability}, resp. Proposition \ref{prop:projIrred}), and the existence of open small sets and strong aperiodicity (Section \ref{subsub:oneptsmallaper}, resp. Section \ref{subsub:projsmallaper}). 
\item A positive Lyapunov exponent for $f^n_\uo$ was checked in Proposition \ref{prop:LyapPierre}, so condition  \ref{suffcond2} is met. 
\item Geometric ergodicity for the two-point process as in  \ref{suffcond4} follows from Harris's Theorem \ref{thm:Harris}: irreducibility was checked in Proposition \ref{prop:twoPointIrred} and
the small set and aperiodicity properties in Section \ref{subsub:twoptsmallaper}, while the drift condition follows from  \ref{suffcond0}-- \ref{suffcond3}. 
\end{itemize}
With  \ref{suffcond0}-- \ref{suffcond4} in place, Proposition \ref{prop:twoPointToMixing} applies, completing the proof of Theorem \ref{thm:main}. 


\appendix

\section{Proof of Theorem \ref{thm:Harris}}\label{app:harris}

Below we sketch how Theorem \ref{thm:Harris} can be reduced from results in the book  \cite{meyn2012markov}. 
The following is a brief sketch of the necessary definitions and basic facts, interspersed with supplementary arguments not found in \cite{meyn2012markov}. Throughout, 
$P$ is the transition kernel of a Markov chain on a complete metric space $X$, not necessarily Feller unless otherwise stated. Recall that $\mathcal M(X)$ is the space of Borel probability measures on $X$.

\subsection{$T$-chain property}
We begin from the following definitions.
\begin{defn}\label{defn:sampleKernel}
Let $a = \{ a(n)\}_{n \geq 1}$ be a {\it sample distribution}, i.e., $a(n) \geq 0$ for all $n$ and $\sum_{1}^\infty a(n) = 1$. The {\it sample Markov kernel} $K_a(x, \cdot)$ is defined by
\begin{equation}
K_a(x, \cdot) = \sum_{n = 1}^\infty a(n) P^n(x, \cdot) \,, \quad x \in X \,. 
\end{equation}
\end{defn}
Observe that with $P$ as above, we have that $K_a(x, \cdot) \in \mathcal M(X)$ for all $x \in X$. 

\begin{defn}\label{defn:Tchain}
We say that $P$ is a {\it $T$-chain} if there is an assignment to each $x \in X$ of a 
finite measure $T(x, \cdot)$ on $X$ such that:  
\begin{enumerate} [label=(\roman*), ref=(\roman*)]
 \item \label{tchain1} for all  $A \in {\rm Bor}(X)$ we have $x \mapsto T(x, A)$ is 
 lower semi-continuous, i.e., if $x_n \to x$ then $\liminf_n T(x_n, A) \geq T(x, A)$;
 \item\label{tchain2}
 there exists a sample distribution $a$ so that $K_a(x, \cdot) \geq T(x, \cdot)$ for all $x \in X$; and 
 \item\label{tchain3} $T(x, X) > 0$ for all $x \in X$. 
 \end{enumerate}
\end{defn}

\begin{rmk}\label{rmk:strongFeller}
The $T$-chain property is a weakening of the strong Feller property, which is equivalent to the continuity of $x \mapsto P(x, \cdot)$ in the TV norm on $\mathcal M(X)$ \cite{seidler}. 
The additional flexibility of the $T$-chain property is quite useful in applications. For instance, it is straightforward to check from the definitions that for any $n \geq 1$, we have that if the iterated kernel $P^n$ is a $T$-chain, then $P$ is a $T$-chain. The same is not true for the strong Feller property. 
\end{rmk}

We use the following sufficient condition for the $T$-chain property: 

\begin{prop}\label{prop:TchainSuff}
Assume $P$ is Feller, topologically irreducible and admits an open small set. Then, $P$ is a $T$-chain. 
\end{prop}
\begin{proof}
The following is adapted from Propositions 6.2.3 and 6.2.4 in \cite{meyn2012markov}. 
Assume $A \subset X$ is open and $P^n$-small for some $n \geq 1$ and measure $\nu_n$ on $X$. Fix the sample distribution $a(n) = 2^{-n}, n \geq 1$, and for $x \in X$, define
\begin{align}\label{eq:formTcomponent}
T(x, B) = \nu_n(B) \sum_{m = 1}^\infty a(m) P^{m-n}(x, A)  \,,
\end{align}
using the convention that $P^0(x, \cdot) = \delta_x(\cdot)$, the unit Dirac mass at $x$.
By the small set property, $T(x, B) \leq K_a(x, B)$ for all $B \subset X$ measurable, while $T(x, X) > 0$ for all $x \in X$ by topological irreducibility, which ensures $P^N(x, A) > 0$ for some $N = N(x, A) \in \N$. 

Lastly, we check that $x \mapsto T(x, K)$ is lower semi-continuous for all measurable $K \subset X$. For this it suffices to check that $x \mapsto P^N(x, A)$ is lower semi-continuous for all $N$. Note that $P^N(x, A) = P^N \chi_A(x)$, where $\chi_A$ is the indicator function for $A$. Since $A$ is open, $\chi_A$ is lower semi-continuous, and so is the pointwise limit of an increasing sequence of continuous functions $\varphi_n : X \to \R$ (Exercise 4(g), pg. 132 of \cite{stromberg2015introduction}); without loss, the $\varphi_n$ may be taken to be nonnegative. By the monotone convergence theorem, we have that $P^N \chi_A$ is the pointwise increasing limit of $P^N \varphi_n$ as $n \to \infty$. By the Feller property, $P^N \varphi_n$ is continuous for each $n$, and so we conclude that $P^N\chi_A$ is the pointwise limit of an increasing sequence of continuous functions, hence lower semi-continuous. 
\end{proof}

\subsection{$\psi$-irreducibility} 
 Let $\phi$ be a measure on $X$. The kernel $P$ is said to be $\phi$-{\bf irreducible} 
if, whenever $\phi(K) > 0$ for some measurable $K \subset X$, we have that 
for all $x \in X$ there exists some $n \geq 1$ such that $P^n(x, K) > 0$.

Now, consider the sample distribution $a(n) = 2^{-n}, n \geq 1$. By  \cite{meyn2012markov}*{Proposition 4.2.2}, if $P$ is $\phi$-irreducible then the measure
\begin{align}\label{eq:explicitPsi}
\psi(A) = \int_X \phi(\dd x) K_a(x, A)
\end{align}
is such that $P$ is $\psi$-irreducible, and moreover, it is \emph{maximal} in
the sense that 
\begin{itemize}
	\item[(i)] For any other measure $\phi'$ on $X$, $P$ is $\phi'$ irreducible iff 
	$\phi' \ll \psi$; and
	\item[(ii)] If $\psi(A) = 0$, then $\psi \{ x \in X : P^n(x, A) > 0 \text{ for some }n\} = 0$. 
\end{itemize}
From here on, we adopt the convention in \cite{meyn2012markov} of referring to $P$
as $\psi$-irreducible when $\psi$ satisfies the properties (i), (ii) above. 
The $T$-chain property allows us to check $\psi$-irreducibility as follows: 
\begin{lem}\label{lem:TchainPsiIrred6}
If $P$ is both $T$-chain and topologically irreducible, 
\begin{itemize}
\item[(a)] it is $\psi$-irreducible; and
\item[(b)] $\psi$ is locally positive, i.e., $\psi(U) > 0$ for all nonempty open $U \subset X$. 
\end{itemize}
\end{lem}
\begin{proof}
Item (a) is Proposition 6.2.2 of \cite{meyn2012markov}. Indeed, the proof given there implies $P$ is $\phi$-irreducible with $\phi (\cdot)= T(x, \cdot)$ for any $x \in X$, and $\psi$ can be taken to be given by \eqref{eq:explicitPsi}. Item (b) follows from topological irreducibility. 
\end{proof}

The following Corollary is useful in our arguments to come on periodicity. 

\begin{cor}\label{cor:irredMeasureCompare}
Suppose $P$ is Feller, topologically irreducible, and admits an open small set. Moreover, assume that for some $d \geq 1$ we that $P^d$ is topologically irreducible. Then, $P$ is $\psi$-irreducible and $P^d$ is $\psi^d$-irreducible, and $\psi^d \ll \psi$. 
\end{cor}

\begin{proof}
Fix an open, $\nu_n$-small set $A$ and let $n' \geq 0$ such that $n + n' = m d$ for some integer $m \geq 1$. Observe that $A$ is also a $\nu_{m d}$-small set with $\nu_{m d}(\cdot) := (P^{n'})^* \nu_n = \int P^{n'}(x, \cdot) d \nu_{n}(x)$. Apply Proposition \ref{prop:TchainSuff} and Lemma \ref{lem:TchainPsiIrred6} to both $P$ and $P^d$ to conclude each is $\psi$ (resp. $\psi^d$)-irreducible for some maximal irreducibility measure $\psi$ (resp. $\psi^d$). Fixing an arbitrary $x_0 \in X$, using  \eqref{eq:formTcomponent} 
to build the continuous component $T$ out of the $\nu_{md}$-small set $A$, and 
using \eqref{eq:explicitPsi} to identify a maximal irreducibility measure, we can take
\[
\psi(\cdot) = 2^{m d - 1} \sum_{k = md}^\infty \sum_{\ell = 1}^\infty 2^{- k - \ell} 
P^{k-md} (x_0, A) \int \nu_{m d} (\dd x) P^{\ell}(x, \cdot) 
\]
and
\[
\psi^d(\cdot) := 2^{m - 1}\sum_{k = m}^\infty \sum_{\ell = 1}^{\infty} 2^{- k - \ell} 
P^{(k-m)d} (x_0, A) \int \nu_{m d} (\dd x) P^{\ell d}(x, \cdot) \, . 
\]
Note that $\psi^d \ll \psi$, as desired. Note that by maximality, it follows that $\hat \psi^d \ll \hat \psi$ holds for all maximal irreducibility measures $\hat \psi$ for $P$ and $\hat \psi^d$ for $P^d$. 
\end{proof}

\subsection{Petite sets}

Let $a$ be a sample distribution and let $\nu_a$ be a nontrivial measure on $X$. We say that a set $A \subset X$ is $\nu_a$-\emph{petite} if 
\[
K_a(x, B) \geq \nu_a(B)
\]
for all measurable $B \subset X$ and for all $x \in A$. The $T$-chain property can be used to check petiteness: 
\begin{lem}[Theorem 6.2.5 (ii) in \cite{meyn2012markov}]\label{lem:compactIsPetite6}
If $P$ is a $\psi$-irreducible $T$-chain, then every compact set is petite. 
\end{lem}

We note that a small set is petite, but the converse is not true in general. 

\subsection{Periodicity}
The following result characterizes \emph{periodicity} of $\psi$-irreducible Markov chains. 
\begin{lem}[Theorem 5.4.4 in \cite{meyn2012markov}]\label{lem:defnPeriod}
Assume $P$ is $\psi$-irreducible and there exists a $\nu_n$-small set $C \subset X$ with $\psi(C) > 0$. Then, there exists an integer $d \geq 1$ and a disjoint collection of measurable sets $D_1, \cdots, D_d$ (a ``$d$-cycle'') such that: 
\begin{itemize}
	\item[(a)] For all $x \in D_i, i \in \{ 1, \cdots, d\}$, we have $P(x, D_{i + 1}) = 1$. 
	Here, we follow the convention that $D_{i + kd} := D_i$ for all $k \in \Z$.
	\item[(b)] We have $\psi(X \setminus \cup_{i =1 }^d D_i) = 0$.
\end{itemize}
Moreover, the collection $\{ D_i\}$ is \emph{maximal} in the sense that 
if $d' \geq 1, \{ D_i'\}_{i = 1}^{d'}$ is another collection satisfying (a) and (b) above, then $d'$ cuts $d$, while if $d = d'$ then up to a reordering of the $i$, we have 
that $D_i = D_i'$ up to $\psi$-null sets. 
\end{lem}

When $d = 1$, we call $P$ \emph{aperiodic}, while if $d > 1$ then $d$ is called the \emph{period} of $P$. The following provides a useful relationship between $d$-cycles and small sets. 

\begin{lem}\label{lem:containSmallSet}
Suppose $P$ is $\psi$-irreducible and admits a $\nu_n$-small set $A$ for which $\psi(A) > 0$. Assume $P$ is periodic of period $d \geq 1$ and let $\{ D_i\}$ be a $d$-cycle. Then there exists $i \in \{ 1, \cdots, d\}$ such that $\psi(A \setminus D_i) = 0$, i.e., $A \subset D_i$ up to a $\psi$-null set.  
\end{lem}

\begin{proof}
Suppose, for the sake of contradiction, that $\psi(A \cap D_\ell) \neq 0$ for $\ell = i, j \in \{ 1, \cdots, d\}, i \neq j$. Then, for all $x \in A \cap D_j$, we have $P^n(x, D_{i + n}) = 0$, hence $\nu_n(D_{i + n}) = 0$. 
On the other hand, for all $x \in A \cap D_i$, we have 
\begin{align}
1 = P^n(x, D_{i + n}) = P^n(x, D_{i +n}) - \nu_n(D_{i + n}) \leq 1 - \nu_n(X) < 1 \,. 
\end{align}
This is a contradiction. Therefore $\psi(A \cap D_i) \neq 0$ for at most one $i \in \{ 1 , \cdots, d\}$. As $\psi(X \setminus \cup_i D_i) = 0$, the conclusion follows. 
\end{proof}

The following can be used to rule out periodicity. 

\begin{lem}\label{lem:suffCondAperiodicity}
Assume $P$ is Feller, topologically irreducible, admits an open small set, and that there exists $x_\star \in X$ such that $P(x_\star, U) > 0$ for all open $U \ni x_\star$ (what we call ``strong aperiodiciy'' in the statement of Theorem \ref{thm:Harris}). Then, $P$ is aperiodic. 
\end{lem}

The proof uses the following Claim which we prove first. 

\begin{cla}
Under the assumptions of Lemma \ref{lem:suffCondAperiodicity}, we have that 
$P^n$ is topologically irreducible for all $n \geq 1$. 
\end{cla}
\begin{proof}[Proof of Claim]
Fix $n > 1$, an open $U \subset X$ and a point $x_0 \in X$. To start, let $N = N(x_\star, U)$ be such that $P^N(x_\star, U) > 0$. Since $P^N$ is Feller and $U$ is open, the mapping $x \mapsto P^N(x, U)$ is lower semicontinuous (c.f. the proof of Proposition \ref{prop:TchainSuff}) and so there is an open neighborhood $U_\star$ containing $x_\star$ such that $P^N(x, U) \geq \frac12 P^N(x_\star, U) > 0$ for all $x \in U_\star$. Similarly, as $P^N(x_\star, U_\star) > 0$ by hypothesis, there is an open neighborhood $U'_\star \subset U_\star$ containing $x_\star$ such that $P^i(x, U_\star) \geq \frac12 P^i(x_\star, U_\star) > 0$ for all $i \in \{ 0,\cdots, n-1\}$. Finally, let $N_\star = N(x_0, U_\star')$ be such that $P^{N_\star}(x_0, U_\star') > 0$ and observe that
\[
P^{N + N_\star + i} (x_0, U) \geq \int_{\substack{y \in U_\star \\ y' \in U_\star'}} P^{N_\star}(x_0, dy') P^i(y', dy) P^{N}(y, U) > 0 \,. 
\]
To conclude, note that there is some $i \in \{0,\cdots, n-1\}$ such that $n | (N + N_\star + i)$. 
\end{proof}

\begin{proof}[Proof of Lemma \ref{lem:suffCondAperiodicity}]
By Proposition \ref{prop:TchainSuff}, $P$ is a $T$-chain, hence $\psi$-irreducible by Lemma \ref{lem:TchainPsiIrred6}, with $\psi(U) > 0$ for all open $U \subset X$. 
For the sake of contradiction, assume the period $d$ is $> 1$. 
By Corollary \ref{cor:irredMeasureCompare} and the Claim, $P^d$ is likewise 
$\psi^d$-irreducible with $\psi^d \ll \psi$. 

%
%

Let $\{ D_1, \cdots, D_d\}$ denote a $d$-cycle for $P$ (Lemma \ref{lem:defnPeriod}). 
By Lemma \ref{lem:containSmallSet} we have $\psi(A \setminus D_i) = 0$ for some $i \in \{1,\cdots, d\}$,  hence $\psi^d(A \setminus D_i) =0$. From this, we obtain $\psi^d(A \cap D_i) = \psi^d(A) > 0$. As $\psi_d$ is an irreducibility measure for $P^d$, we have that for all $x \in X$ there exists $ k \geq 1$ with $P^{k d}(x, D_i \cap A) > 0$. 
On the other hand, by periodicity we have that $P^{k d}(x, D_j \cap A) \leq P^{kd}(x, D_j) = 0$ for all $x \in D_{j-1}$, $k \geq 1$ and $j \in \{ 1,\cdots, d\}$; this is a contradiction at $j = i$. 
\end{proof}

\subsection{Harris Theorem}

Our version of Harris theorem as stated in Theorem \ref{thm:Harris}
 will be deduced from the following. 
\begin{prop}[Theorem 16.1.2 of \cite{meyn2012markov}]\label{prop:harrisFromMT}
Assume that $P$ is $\psi$-irreducible and aperiodic, and moreover, that 
there is some petite set $C \subset X$ and a function $V : X \to [0,\infty)$ satisfying the drift condition
\[
P V \leq \alpha V + b \chi_C
\]
for some $\alpha \in (0,1), b > 0$. Then, $P$ is $V$-uniformly geometrically ergodic. 
\end{prop}
\begin{proof}[Proof of Theorem \ref{thm:Harris}]
Recall the hypotheses of Theorem \ref{thm:Harris}: $P$ is Feller, and moreover, (a) $P$ admits an open small set; (b) is topologically irreducible; (c) there exists $x_\star \in X$ such that $P(x_\star, U) > 0$ for all $U \ni x_\star$ open; and (d) the drift condition $P V \leq \alpha V + b \chi_C$ holds for some $V : X \to [1,\infty)$, where $\alpha \in (0,1), b > 0$ and $C \subset X$ is compact. 

By (a) and (b), $P$ is a $T$-chain (Proposition \ref{prop:TchainSuff}). By Lemma \ref{lem:TchainPsiIrred6}, the $T$-chain property and topological irreducibility imply $P$ is $\psi$-irreducible. Lemma \ref{lem:suffCondAperiodicity} implies $P$ is aperiodic. 
By Lemma \ref{lem:compactIsPetite6}, every compact set is petite. So, the drift condition as stated in Theorem \ref{thm:Harris}\ref{item:drift} implies that in Proposition \ref{prop:harrisFromMT}.
\end{proof}

\section{Proofs of Proposition \ref{prop:bougerolRegularity} and Corollary \ref{cor:bougerol}}\label{app:bougerol}

Recall that $f^n_\uo$ is a continuous RDS over a compact metric space $X$ with transition kernel $P$, assumed to be uniformly geometrically ergodic with stationary measure $\pi$. Moreover, $x \mapsto \Ac_{\omega, x}$ is a continuous linear cocycle of $d \times d$-matrices satisfying the integrability assumption of Theorem \ref{thm:MET}. Write $\lambda_1(\Ac) > \cdots > \lambda_r(\Ac)$ for the Lyapunov exponents of $\Ac^n_{\uo, x}$ and $\lambda_\Sigma(\Ac)$ for the summed Lyapunov exponent. 

\subsection{Proof of Proposition \ref{prop:bougerolRegularity}}

We seek to show that if 
\[
d \lambda_1(\Ac) = \lambda_\Sigma(\Ac) \, ,
\]
then there is a weak$^*$ continuous family $(\nu_x)_{x \in \Supp(\pi)}$
	with the property that for $\P \times \pi$-almost every $(\omega, x) \in \Omega_0 \times X$, we have that
	\begin{align}\label{eq:invar6}
	\nu_x = (\Ac_{\omega, x}^n)^T \nu_{f_\omega x} \, .
	\end{align}	
Equivalently, we can consider the cocycle $\Bc^n_{\uo, x} = \Bc_{\omega_n, f^{n-1}_\uo x} \circ \cdots \circ \Bc_{\omega_1, x}$ 
generated by $\Bc : \Omega_0 \times X \to GL_d(\R)$, 
\[
\Bc_{\omega, x} = (\Ac_{\omega, x})^{-T} \, , \quad \Bc^n_{\uo, x} = (\Ac^n_{\uo, x})^{-T} \,. 
\]
With this notation, it suffices to find a weak$^*$ continuous family $(\nu_x)$ such that
\begin{align}\label{eq:furst6}
\Bc^n_{\uo, x} \nu_x = \nu_{f^n_\uo x}
\end{align}
for $\P \times \pi$-a.e. $(\uo, x)$ and for all $n \geq 1$. 
Theorem \ref{thm:MET} applies to the cocycle $\Bc^n_{\uo, x}$, yielding Lyapunov exponents $\lambda_i(\Bc)$, $1 \leq i \leq r'$ and a summed Lyapunov exponent $\lambda_\Sigma(\Bc)$. 

\begin{lem}\label{lem:LEagree6}
We have that $r = r'$ and 
\[
	\lambda_i(\Bc) = \lambda_{r - i+1} (\Ac)
\]
for all $1 \leq i \leq r$. Additionally, $\lambda_\Sigma(\Bc) = \lambda_\Sigma(\Ac)$. 
\end{lem}
\begin{proof}
	By Remark \ref{rmk:sumLE3}, the Lyapunov exponents $\lambda_i(\Ac)$ are the distinct values among the quantities $\chi_i(\Ac)$, 
	\[
	\chi_i(\Ac) := \lim_{n \to \infty} \frac1n \log \sigma_i(\Ac^n_{\uo, x}) \, , 
	\]
	while the multiplicity $m_i(\Ac)$ of the $i$-th Lyapunov exponent $\lambda_i(\Ac)$ is given by 
	\[
	m_i = \# \{ 1 \leq j \leq d : \chi_j(\Ac) = \lambda_i(\Ac) \} \,. 
	\]
The same holds for the limits $\chi_i(\Bc)$, exponents $\lambda_i(\Bc)$ and multiplicities $m_i(\Bc)$. Indeed, we have that
\[
\chi_i(\Bc) = \lim_{n \to \infty} \frac1n \log \sigma_i(\Bc^n_{\uo, x}) 
= - \lim_{n \to \infty} \frac1n \log \sigma_{d - i+1} (\Ac^n_{\uo, x})
\]
on using that for an invertible matrix $A$, we have that
\[
\sigma_i(A^{-T}) = \frac{1}{\sigma_{d - i + 1} (A)} \,. 
\]
The desired conclusion follows. 
\end{proof}

In view of Lemma \ref{lem:LEagree6}, we see that $d \lambda_1(\Ac) = \lambda_\Sigma(\Ac)$ holds iff $d \lambda_1(\Bc) = \lambda_\Sigma(\Bc)$. 
The latter implies the existence of a \emph{measurable} family $(\nu_x)$ 
such that \eqref{eq:furst6} holds for $\P \times \pi$-almost every $(\uo, x)$ and 
for all $n \geq 1$. 
It remains to show that we can construct a weak$^*$-continuous family of measures $(\hat \nu_x)$ agreeing with $(\nu_x)$ $\pi$-almost everywhere. 
For this we will prove the following slightly stronger result, which holds irrespective of whether or not $d \lambda_1(\Ac) = \lambda_\Sigma(\Ac)$. 

\begin{lem}\label{lem:ctyBackwardsRelation}
Assume that $(\nu_x)$ is a measurable family of measures on $S^{d-1}$ and that the 
relation
\begin{align}\label{eq:backwardsInvariance6}
\nu_x = \E (\Ac_{\uo, x}^n)^T_* \nu_{f^n_\uo x}
\end{align}
holds for $\pi$-almost all $x \in X$ and for all $n \geq 1$. 
Then, there exists a weak$^*$ measurable family $(\hat \nu_x)_{x \in \Supp(\pi)}$ satisfying the invariance relation \eqref{eq:backwardsInvariance6} and 
for which 
\[
\nu_x = \hat \nu_x \quad \text{ for } \pi-\text{a.e. } x \in X \,. 
\]
\end{lem}

To complete the proof of Proposition \ref{prop:bougerolRegularity}, 
we observe that if $d \lambda_1(\Ac) = \lambda_\Sigma(\Ac)$, then \eqref{eq:backwardsInvariance6} holds pointwise almost-surely in $\uo$ without 
the expectation $\E$ for the measurable family $(\nu_x)$, hence \eqref{eq:backwardsInvariance6} is immediate and weak$^*$ continuity follows. 

\begin{proof}[Proof of Lemma \ref{lem:ctyBackwardsRelation}]

Some notation: given $\varphi \in C(P^{d-1}, \R)$ let $G_\varphi : X \times GL_d(\R) \to \R$ denote the function $G_\varphi(x, A) = \int \varphi(A v) \dd \nu_x$. 
Given $G : X \times GL_d(\R) \to \R$,
define the function $R^n G : X \to \R, n \geq 1,$ by
\[
R^n G(x) :=  \E G(f^n_\uo x, (\Ac^n_{\uo, x})^T) \,. 
\]
Lastly, we define $g_\varphi(x) := \int \varphi \dd \nu_x$. 

By a standard density argument, to prove Lemma \ref{lem:ctyBackwardsRelation} it suffices to prove the following. 
\begin{cla}
For each $\varphi \in C(X, S^{d-1})$ there is a continuous function 
 $\bar g_\varphi : \Supp \pi \to \R$ such that $g_\varphi = \bar g_\varphi$
 holds $\pi$-almost everywhere. 
\end{cla}


To prove the claim, fix $\varphi \in C(X, S^{d-1})$. To start, by Lusin's theorem and the weak$^*$ compactness of the space of probability measures on $S^{d-1}$, there is an increasing sequence of compact subsets $C_1 \subset C_2 \subset \cdots$ of $X$ such that for each $n$, (i) we have that $x \mapsto \nu_x$ is continuous along $x \in C_n$ and (ii) we have $\pi(C_n) \geq 1 - 1/n$. 

Observe that for $G := G_\varphi$, we have that $G|_{C_n \times GL_d(\R)}$
is continuous. By the Tietze extension theorem, for each $n$ there is a function 
$G_n : X \times GL_d(\R) \to \R$ such that 
\[
G|_{C_n \times GL_d(\R)} = G_n|_{C_n \times GL_d(\R)}
\]
and such that $\| G_n\|_\infty \leq \| G\|_{\infty}$. Lastly, define 
$r_n$ to be the indicator function of $C_n \times GL_d(\R)$ and 
observe that $r_n G = r_n G_n$ holds. 

Now, let $g = g_\varphi$. Let $X' \subset X$ be the $\pi$-full measure set along which \eqref{eq:backwardsInvariance6} holds, and observe
that for $x \in X'$, we have
\begin{align*}
|g(x) - R^n G_n(x, \Id)| &= |R^n G(x) - R^n G_n(x)| \\
& \leq |R^n [r_n G](x) - R^n [r_n G_n] (x)| + 
2 \| G \|_\infty |R^n (1 - r_n) (x) | \\
& = 2 \| G \|_\infty |R^n (1 - r_n) (x) | \, .
\end{align*}
Above, we have abused notation somewhat and treated $r_n$ as a function on $X \times GL_d(\R)$. Since $r_n$ has no dependence on the $GL_d(\R)$ coordinate, we see that
\[
R^n (1 - r_n) (x) = P^n \chi_{C_n^c}(x)
\]
and that by geometric ergodicity, $P^n \chi_{C_n^c} \leq (1/n) + C \gamma^n$. 
Therefore $g|_{X'}$ is a uniform limit of continuous functions $R^n G_n|_{X'}$. 
Since $X' \subset \Supp \pi$ is dense, it follows that there is a continuous version 
$\bar g : \Supp \pi \to \R$ of $g$ agreeing up to $\pi$-null sets. 
\end{proof}

\subsection{Proof of Corollary \ref{cor:bougerol}}

Let $(\nu_x)_{x \in \Supp(\pi)}$ denote the weak$^*$ continuous family 
of measures on $S^{d-1}$ such that \eqref{eq:invar6} holds. We seek to show now that for all $x \in \Supp (\pi), n \geq 1$ and for all $(y, A)$ in the support of the law of $(f^n_\uo x, \Ac^n_{\uo, x})$, we have that 
\begin{align}\label{eq:topInvarRelation6}
\nu_x = (A^T)_* \nu_y \,. 
\end{align}
Fix such a pair $(y, A)$. By the definition of topological support, there are sequences $\{\uo^{(m)}\} \subset \Omega, \{ x^{(m)}\} \subset X$ such that $(f^n_{\uo^{(m)}} x^{(m)}, \Ac^n_{\uo^{(m)}, x}) \to (y, A)$ as $m \to \infty$ (in particular, $x^{(m)} \to x$). As \eqref{eq:invar6} holds with full probability, without loss we may choose our sequence $\{ \uo^{(m)}\}$ so that \eqref{eq:invar6} holds for each $(\uo^{(m)}, x^{(m)})$.  
Equation \eqref{eq:topInvarRelation6} now follows on taking the limit $m \to \infty$ in the relation
\[
\nu_{x^{(m)}} = (\Ac^n_{\uo^{(m)}, x^{(m)}})^T_* \nu_{f^n_{\uo^{(m)}} x^{(m)}}
\]
and using that $x \mapsto \nu_x$ is weak$^*$ continuous. 

\section{The Weyl Equidistribution Theorem}\label{app:Weyl}

We use at several points the following ``uniform'' version of the Weyl Equidistribution theorem in $d \geq 1$ dimensions. Below, $\T^d$ is parametrized by $[0, 2\pi)^d$ with addition defined modulo $2 \pi$. 

\begin{defn}
We say that $\zeta = (\zeta^1, \cdots, \zeta^d) \in \R^d$ is \emph{rationally independent} if the only solution $a = (a_1, \cdots, a_d) \in \mathbb Q^d, b \in \mathbb Z$ to the equation
\[
a_1 \zeta^1 + \cdots + a_d \zeta^d = b
\]
is the trivial solution $a = (0, \cdots, 0), b = 0$. 
\end{defn}
Given $\zeta \in \R^d$, we define the mapping $T : \T^d \circlearrowleft, t \in \R$, by 
\[
T_\zeta(x) = x + 2 \pi \zeta \,, 
\]
where as usual the right-hand side is considered modulo $2 \pi$ in each coordinate. 

\begin{thm}[Theorem 6.18 in \cite{Wal82}]
If $\zeta$ is rationally independent, then $T_\zeta$ is uniquely ergodic with invariant measure $m = \Leb_{\T^d}$. 
\end{thm}

\begin{cor}[``Uniform'' Weyl Equidistribution Theorem]\label{cor:weyl}
For any $\zeta \in \R^d$ rationally independent and for any $\eps > 0$, there exists $N = N(\zeta, \eps)$ such that the following holds: for any $x, y \in \T^d$ there exists $n = n(x,y,\zeta,\eps) \leq N(\zeta, \eps)$ such that 
\[
d_{\T^d}(T^n_\zeta (x), y) < \eps \,. 
\]
\end{cor}
\begin{proof}
By Theorem 6.19(i) in \cite{Wal82}, unique ergodicity of $T_\zeta$ implies that 
for all continuous $h : \T^d \to \R$, we have that
\[
S_n h := \frac1n \sum_{i = 0}^{n-1} h \circ T^i_\zeta \to \int h \dd x
\]
converges uniformly on $\T^d$. Fix $\eps > 0$ and let $h_1, \cdots, h_M$ be a 
smooth, nonnegative partition of unity with $\Supp (h_i) \subset B_{\eps/2}(x_i)$ for some collection of centers $\{x_1, \cdots, x_M\} \subset \T^d$. Let $N = N(\zeta, \eps)$ be sufficiently large so that for all $1 \leq i \leq M$, we have that $|S_n h_i - \int h_i \dd x| > \frac12 \int h_i \dd x$ uniformly on $\T^d$. 

Fix now $x, y \in \T^d$. Observe that $y \in B_{\eps/2}(x_i)$ for some $i \in \{ 1, \cdots, M\}$, and that $S_N h_i (x) > \frac12 \int h_i \dd x > 0$. In particular, $h_i \circ T^n_\zeta(x) > 0$ for some $n = n(x,y,\zeta, \e)$. Since both $y$ and $T^n_\zeta(x)$ lie in $B_{\eps/2}(x_i)$, we conclude $d_{\T^d}(T^n_\zeta(x), y) < \eps$. 
\end{proof}


\bibliographystyle{abbrv}
\bibliography{bibliography}

\end{document}